\documentclass[dvipdfmx]{article}
\usepackage{amsmath,amssymb,mathrsfs,amsthm}
\usepackage{array,enumerate}
\usepackage{latexsym}
\usepackage{mathrsfs}
\usepackage{amscd}
\usepackage[all]{xy}
\usepackage{graphicx}
\usepackage[svgnames]{xcolor}
\usepackage{tikz}
\usetikzlibrary{intersections}
\theoremstyle{definition}
\newtheorem{thm}{Theorem}[section]
\newtheorem{dfn}[thm]{Definition}
\newtheorem{note-dfn}[thm]{Notation-Definition}
\newtheorem{note-prop}[thm]{Notation-Propositoin}

\newtheorem{exam}[thm]{Example}
\newtheorem{prop}[thm]{Proposition}

\newtheorem{lem}[thm]{Lemma}
\newtheorem{rem}[thm]{Remark}

\newtheorem{prop-dfn}[thm]{Proposition-Definition}

\newtheorem{claim}[thm]{Claim}

\newcommand{\Z}{\mathbb{Z}}
\newcommand{\Q}{\mathbb{Q}}

\newcommand{\ZZ}{\widehat{\mathbb{Z}}}
\newcommand{\F}{\mathbb{F}}
\newcommand{\Spec}{\mathrm{Spec}\,}
\newcommand{\sep}{\mathrm{sep}}
\newcommand{\deq}{\overset{\mathrm{def}}{=}}
\title{Criteria for good reduction of hyperbolic polycurves}
\author{Ippei Nagamachi}

\begin{document}

\maketitle

\begin{abstract}
 We give good reduction criteria for hyperbolic polycurves, i.e.,
successive extensions of families of curves, under some assumptions.
These criteria are higher dimensional versions of the good reduction criterion for hyperbolic curves given by Oda and Tamagawa.
\end{abstract}

\tableofcontents

\section{Introduction}
Let $K$ be a discrete valuation field, $O_{K}$ the valuation ring of $K$, $p\,(\geq 0)$ the residual characteristic of $K$, $K^{\rm sep}$ a separable closure of $K$, $G_K$ the absolute Galois group ${\rm Gal}(K^{\rm sep}/K)$ of $K$, and $I_K$ an inertia subgroup of $G_K$.
(Note that $I_K$, as a subgroup of $G_K$, depends on the choice of a prime ideal in the integral closure of $O_K$ in $K^{\rm sep}$ over the maximal ideal of $O_K$, but it is independent of this choice up to conjugation.)
Let $X$ be a proper smooth variety over $K$.
$X$ is said to have good reduction if there exists a proper smooth scheme $\mathfrak{X}$ over $O_{K}$ whose generic fiber is isomorphic to $X$ over $O_{K}$.
(Such a scheme $\mathfrak{X}$ is called a smooth model of $X$.)
In arithmetic geometry, it is an important to know criteria to determine whether $X$ has good reduction.
Various criteria for good reduction in terms of Galois representations have been established for certain classes of varieties.
N\'eron, Ogg, and Shafarevich established a criterion in the case of elliptic curves, and Serre and Tate generalized this criterion to the case of abelian varieties \cite{ST}.
Their criterion claims that an abelian variety has good reduction if and only if the action of $I_{K}$ on the first $l$-adic \'etale cohomology group of $X\times_{\Spec K}\Spec K^{\sep}$ is trivial for some prime number $l \neq p$.

As a non-abelian version of the above result, Oda showed that a proper hyperbolic curve has good reduction if and only if the outer action of $I_{K}$ on the pro-$l$ fundamental group of  $X\times_{\Spec K}\Spec K^{\sep}$ is trivial (cf.\,\cite{Oda1} and \cite{Oda2}).
To state Oda's result precisely, we fix some notations.
For a profinite group $G$ and $p$ as above (resp.\,a prime number $l$), we denote the pro-$p'$ (resp.\,pro-$l$) completion of $G$, which is defined to be the limit of the projective system of quotient groups of $G$ with finite order prime to $p$ (resp.\,with finite $l$-power order) by $G^{p'}$ (resp.\,$G^{l}$).
Here, if $p=0$, the order of every finite group is considered to be prime to $0$.

Suppose that $X$ is a proper hyperbolic curve (i.e., a geometrically connected proper smooth curve of genus $\geq 2$) over $K$. 
Then the pro-$l$ completion $\pi_1(X \times_{\Spec K}\Spec K^{\rm sep}, \overline{t})^{l}$ of the \'etale fundamental group $\pi_1(X \times_{\Spec K}\Spec K^{\rm sep}, \overline{t})$ (with a base point $\overline{t}$) admits a continuous homomorphism 
\begin{equation}
\begin{split}
\rho: G_K \rightarrow &{\rm Out}(\pi_1(X \times_{\Spec K}\Spec K^{\rm sep}, \overline{t})^{l})\\
:=  &{\rm Aut}(\pi_1(X \times_{\Spec K}\Spec K^{\rm sep}, \overline{t})^{l})/ {\rm Inn}(\pi_1(X \times_{\Spec K}\Spec K^{\rm sep}, \overline{t})^{l}).
\end{split}
\end{equation}
We refer to this outer representation as the outer Galois representation associated to $X$.
Here, ${\rm Aut}(\pi_1(X \times_{\Spec K}\Spec K^{\rm sep}, \overline{t})^{l})$ (resp.\,${\rm Inn}(\pi_1(X \times_{\Spec K}\Spec K^{\rm sep}, \overline{t})^{l})$) is the group of continuous automorphisms of the profinite group $\pi_1(X \times_{\Spec K}\Spec K^{\rm sep}, \overline{t})^{l}$ (resp.\,the group of inner automorphisms of the profinite group $\pi_1(X \times_{\Spec K}\Spec K^{\rm sep}, \overline{t})^{l}$).
Oda proved that $X$ has good reduction if and only if the restriction of $\rho$ to ${I_K}$ is trivial.
Tamagawa generalized this criterion to (not necessarily proper) hyperbolic curves \cite{Tama1} (cf.\,Definition \ref{reddef}).

Oda and Tamagawa's criterion can be regarded as a result in anabelian geometry.
Indeed, a hyperbolic curve is a typical example of an anabelian variety, i.e., a variety which is determined by its outer Galois representation
$G_{K} \rightarrow \text{Out}\, \pi_1(X\times_{\Spec K}\Spec K^{\rm sep}, \overline{t})$ (under a suitable assumption on $K$) (cf.\,\cite{Tama1} and \cite{Moch}).
Therefore, it would be natural to expect that we can obtain informations on the reduction of $X$ from the outer Galois representation associated to $X$.

The class of hyperbolic polycurves, that is, 
varieties $X$ which admit a strucure of successive smooth fibrations (called a sequence of parameterizing morphisms (cf.\,Definition \ref{hyperbolicpoly}))
\begin{equation}\label{0}
X = X_n \overset{f_n}{\rightarrow} X_{n-1} \overset{f_{n-1}}{\rightarrow} \cdots 
\overset{f_2}{\rightarrow} X_1 
\overset{f_1}{\rightarrow} \mathrm{Spec}\,K 
\end{equation}
whose fibers are hyperbolic curves, is considered to be anabelian.
Indeed, the Grothendieck conjecture for hyperbolic polycurves of dimension up to $4$ holds under suitable assumptions on $K$ \cite{Moch} \cite{Ho}. 
Moreover, in the case where $X$ is a strongly hyperbolic Artin neighborhood (of any dimension) (cf.\,{\cite[Definition 6.1]{SS}}) and $K$ is finitely generated over $\Q$, the Grothendieck conjecture for such a variety holds \cite{SS}.
Thus, we can expect that there exists a good reduction criterion for hyperbolic polycurves analogous to that of Oda and Tamagawa.

In \cite{Nag}, we studied a good reduction criterion for proper hyperbolic polycurves under some assumptions.
In this paper, we improve the main theorem of \cite{Nag} and discuss not necessarily proper  cases.
The main results of this paper are as follows:

\begin{thm}
Let $K, O_{K}, p, K^{\rm sep}, G_K,$ and $I_{K}$ be as above.
Let $X$ be a proper hyperbolic polycurve over $K$ and $g_{X}$ the maximum genus of $X$ (cf.\,Definition \ref{hyperbolicpoly}.3).
Consider the following conditions:
\begin{description}
\item{(A)}\mbox{} $X$ has good reduction.
\item{(B)}\mbox{} The outer Galois representation $I_{K} \rightarrow
\mathrm{Out}(\pi _1 (X  \times_{\mathrm{Spec}\, K} \mathrm{Spec}\,K^{\mathrm{sep}} ,\overline{t})^{p'})$ is trivial.
\end{description}
Then we have the following:
\begin{enumerate}
\item (A) implies (B). 
\item If $p = 0$, (B) implies (A).
\item If $p > 2g_{X}+1$ and the dimension of $X$ is $2$, (B) implies (A).
\item Suppose that $p > 2g_{X}+1$, $X$ has a $K$-rational point $x$, and the Galois representation
$I_{K(x)} \to \mathrm{Aut}(\pi _1 (X  \times_{\mathrm{Spec}\, K} \mathrm{Spec}\,K^{\mathrm{sep}} ,\overline{x})^{p'})$ defined as (\ref{3doa}) in Section \ref{repsect} is trivial.
Then (A) holds.
\end{enumerate}
\label{main,thm}
\end{thm}

\begin{thm}
Let $K$, $O _{K}, p, K^{\rm sep}$, and $I_{K}$ be as in Theorem \ref{main,thm}.
Let $X$ be a hyperbolic polycurve over $K$ with a sequence of parameterizing morphisms
 \begin{equation}
\mathcal{S} : X = X_{n} \rightarrow X_{n-1} \rightarrow \ldots \rightarrow X_{1} \rightarrow X_{0} = \mathrm{Spec}\,K.
\end{equation}
Write $b_{\mathcal{S}}$ for the maximal first Betti number of $\mathcal{S}$ (cf.\,Definition \ref{hyperbolicpoly}.3).
Consider the following conditions:

\begin{description}
\item{(A)}\mbox{} There exists a hyperbolic polycurve $\mathfrak{X} \rightarrow \mathrm{Spec}\, O_{K}$ with a sequence of parameterizing morphisms
 \begin{equation}
\mathfrak{X} = \mathfrak{X}_{n} \rightarrow \mathfrak{X}_{n-1} \rightarrow \ldots \rightarrow \mathfrak{X}_{1} \rightarrow \mathfrak{X}_{0} = \mathrm{Spec}\,O_{K}
\end{equation}
whose generic fiber is isomorphic to $(X, \mathcal{S})$ (cf.\,Definition \ref{hyperbolicpoly}.1).

\item{(B)}\mbox{} The outer Galois representation $I_{K} \rightarrow
\mathrm{Out}(\pi _1 (X  \times_{\mathrm{Spec}\, K} \mathrm{Spec}\,K^{\mathrm{sep}} ,\overline{x})^{p'})$ is trivial.
\end{description}
Then we have the following:
\begin{enumerate}
\item (A) implies (B). 
\item If $p = 0$, (B) implies (A).
\item Suppose that $p > b_{\mathcal{S}}+1$ and the dimension of $X$ is $2$.
Then (B) implies (A).
\end{enumerate}
\label{open,main,thm}
\end{thm}

If we assume a very strong condition on $b_{\mathcal{S}}$ and $p$, (B) implies (A) in the case where $\dim X \geq 3$.

\begin{thm}
Let $K$, $O _{K}$, $I_{K}$, $X$, $\mathcal{S}$, $b_{\mathcal{S}}$, and $n$ be as in Theorem \ref{open,main,thm}.
Suppose that $n \geq 3$.
Define a function $f_{b_{\mathcal{S}}}(m)$ for $m \geq 3$ in the following way:
\begin{itemize}
\item
For $m=3$, $f_{b_{\mathcal{S}}}(3) = 2^{b_{\mathcal{S}}^{2}}$.
\item
For $m \geq 3$,
$$f_{b_{\mathcal{S}}}(m+1)= (f_{b_{\mathcal{S}}}(m)) \times (2^{b_{\mathcal{S}}^{2} \times f_{b_{\mathcal{S}}}(m)^{2}})^{f_{b_{\mathcal{S}}}(m)}.$$
\end{itemize}
Consider the conditions (A) and (B) in Theorem \ref{open,main,thm}.
If $p > 2^{b_{\mathcal{S}} \times f_{b_{\mathcal{S}}}(n)}$, (B) implies (A).
\label{higher,main}
\end{thm}

\begin{rem}
The main result of \cite{Nag} is described as follows:
Let $K$, $O _{K}$, and $I_{K}$ be as in Theorem \ref{main,thm}.
Let $X$ be a proper hyperbolic polycurve over $K$ which has a sequence of parameterizing morphisms
$$X=X_{n}\to \ldots \to X_{0}=\Spec K$$
such that, for each $1\leq i \leq n$, $X_{i} \rightarrow X_{i-1}$ has a section.
Write $g_{X}$ for the minimum of the maximal genera of such sequences of parameterizing morphisms of $X$ (cf.\,Definition \ref{hyperbolicpoly}).
Consider the condition (A) in Theorem \ref{main,thm} and the following condition:\\
(B)' Let $x$ be a closed point of $X$, $O_{K(x)}$ a valuation ring  of the residual field $K(x)$ of $x$ over $O_{K}$, $K(x)^{\mathrm{sep}}$ a separable closure of $K(x)$, and $I_{K(x)}$ an inertia subgroup of $O_{K(x)}$ in the absolute Galois group $\mathrm{Gal}(K(x)^{\mathrm{sep}}/K(x))$.
Then the action of $I_{K(x)}$ on $\pi_{1}(X\times_{\mathrm{Spec}\,K}\mathrm{Spec}\,K(x)^{\mathrm{sep}},\overline{x})^{p'}$ is trivial.

\noindent Then (A) implies (B)'.
If $p=0$ or $p > 2g_{X}+1$, (B)' implies (A). 

The condition (B)' is stronger than the condition (B) in Theorem \ref{main,thm} or the condition given in Theorem \ref{main,thm}.4.
Hence, the main result of \cite{Nag} is weaker than Theorem \ref{main,thm} because we need to assume that each $X_{i} \rightarrow X_{i-1}$ has a section and that the condition (B)' is satisfied.
\end{rem}

To prove the implication (B) $\Rightarrow$ (A) or (B)' $\Rightarrow$ (A) by induction on the dimension of $X$, we need a homotopy exact sequence of \'etale fundamental groups of hyperbolic curves.
In the previous paper \cite{Nag}, we constructed homotopy exact sequences of Tannakian fundamental groups of certain categories of smooth $\Q_{l}$-sheaves by using the existence of a section of each morphism $X_{i}\to X_{i-1}$.
Also, we needed the assumption (B)', which is stronger than (B), because we used a criterion for smoothness of $\mathbb{Q}_{l}$-sheaves which is due to Drinfeld \cite{Dri}.

In this paper, we use different arguments from those of \cite{Nag} to obtain stronger results.
Since the implication (A) $\Rightarrow$ (B) follows from a standard specialization argument, we explain key ingredients of the proof of the implication (B) $\Rightarrow$ (A) (assuming the condition on $p$, $g_{X}$, and $b_{\mathcal{S}}$ in the assertions), which enables us to improve the result of \cite{Nag}.
Let $X$ be as in Theorem \ref{main,thm}, Theorem \ref{open,main,thm}, or Theorem \ref{higher,main}.
Take a geometric point of $X\times_{\Spec K}\Spec K^{\sep}$ and write $\Delta$ (resp.\,$\Pi$) for the \'etale fundamental group of the scheme $X\times_{\Spec K}\Spec K^{\sep}$ (resp.\,$X$) determined by this geometric point.

\begin{enumerate}
\item \underline{A decomposition of $\Pi$}

If $p = 0$, we can obtain a decomposition $\Pi \cong \Delta \times Z_{\Pi}(\Delta)$, where $Z_{\Pi}(\Delta)$ is the centralizer subgroup of $\Delta$ in $\Pi$ by using the assumption that the outer Galois action of $I_{K}$ is trivial and the homotopy exact sequences in {\cite[PROPOSITION 2.5]{Ho}} (in the case where $p=0$).
We can prove the implication (B) $\Rightarrow$ (A) by using this decomposition.

\item \underline{Intermediate quotient group}

Note that we do not have appropriate homotopy exact sequences associated to the fibrations $X_{i} \rightarrow X_{i-1}\quad(2\leq i \leq n)$ if $p > 0$.
In fact, the functor of taking pro-$p'$ completion (of profinite groups) is not an exact functor.
Moreover, in the case the characteristic of $K$ is positive, the sequence in {\cite[PROPOSITION 2.5]{Ho}} is no longer exact.
In this paper, we consider an intermediate quotient group of $\Delta$ (which we will write $\Delta^{(l,p')}$ for) between $\Delta^{p'}$ and $\Delta^{l}$, for which we can obtain a homotopy exact sequence.
If the dimension of $X$ is $2$, we can show the implication (B) $\Rightarrow$ (A) by using the center-freeness of $\Delta^{(l,p')}$ and applying an argument similar to that in 1.
If $X$ admits a $K$-rational point $x$, we can use $I_{K(x)}$ instead of $Z_{\Pi}(\Delta)$.
Then we can prove Theorem \ref{main,thm}.4.

 \item \underline{Further intermediate quotient group}

If the dimension of $X$ is equal to or greater than $3$, we do not know whether the group $\Delta^{(l,p')}$ is center-free or not.
However, if $p$ is big enough, we can find a certain quotient $\overline{\Delta}$ of $\Delta$ which is center-free and for which there exists a homotopy exact sequence.
Thus, we can prove the implication (B) $\Rightarrow$ (A) in higher dimensional cases if $p$ is big enough.

 \end{enumerate}

The content of each section is as follows: 
In Section \ref{repsect}, we give a review of outer Galois representations associated to homotopy exact sequences of \'etale fundamental groups of hyperbolic curves.
In Section 3, we give a precise definition of a hyperbolic polycurve and a first step of the proof of Theorem \ref{main,thm} and Theorem \ref{open,main,thm}.
In Section 4, we give a proof of Theorem \ref{main,thm} and Theorem \ref{open,main,thm} in the case of residual characteristic $0$.
In Section \ref{pcase}, we give a proof of Theorem \ref{main,thm} and \ref{open,main,thm} in the case of residual characteristic $p > 0$.
In Section 6, we give a proof of Theorem \ref{higher,main}.
In Section 7, we review the property of extension of family of proper hyperbolic curves proved in \cite{Mor} and prove a non-proper version of it.
In Section 8, we give an example of a hyperbolic polycurve for which the naive analogue of the criterion of Oda and Tamagawa does not hold.
In Section 9, we prove that a sort of specialization homomorphisms of pro-$p'$ \'etale fundamental groups is an isomorphism, which is same as the content of {\cite[Section 7]{Nag2}}.

{\it Acknowledgements:} The author thanks his supervisor Atsushi Shiho for useful discussions and helpful advice.
During the research he was supported by the Program for Leading Graduate Course for Frontier of Mathematical Sciences and Physics. A large percentage of the work was done during his stay at Researches Institute for Mathematical Sciences, Kyoto University.
He thanks them for the hospitality and encouragement, especially Akio Tamagawa and Yuichiro Hoshi.

\label{intro}

\section{Good reduction criterion for hyperbolic curves}

In this section, we recall the good reduction criterion for hyperbolic curves
proven by Oda and Tamagawa.
Let $K, O_{K}, p, K^{\rm sep}, G_K,$ and $I_{K}$ be as in Section \ref{intro}.

\begin{dfn}
Let $S$ be a scheme, $\overline{X}$ a scheme over $S$, and $D$ an effective divisor on $\overline{X}$.
We shall say that the pair $(\overline{X},D)$ is a {\it hyperbolic curve} over $S$ if the following three conditions are satisfied:
\begin{itemize}
\item The morphism $\overline{X} \rightarrow S$ is proper, smooth, with geometrically connected fibers of dimension one and of genus $g$.
\item The morphism $D \rightarrow S$ is finite \'etale of degree $n$.
\item $2g+n-2 > 0.$
\end{itemize}
If $n = 0$ (resp.\,$n > 0$), we call the number $2g$ (resp.\,$2g+n-1$) the first Betti number of the curve.
\label{hyp.curve}
\end{dfn}

\begin{dfn}
\begin{enumerate}
\item
Let $X$ be a proper smooth scheme geometrically connected over $K$.
We say that $X$ has good reduction if there exists a proper smooth $O_{K}$-scheme $\mathfrak{X}$ 
whose generic fiber $\mathfrak{X}\times_{\Spec O_{K}}\Spec K$ is isomorphic to $X$ over $K$. 
We refer to $\mathfrak{X}$ as a {\it smooth model} of $X$.
\item
Let $(\overline{X}, D)$ be a hyperbolic curve over $K$.
We shall say that $(\overline{X}, D)$ has good reduction if there exists a hyperbolic curve $(\overline{\mathfrak{X}}, \mathfrak{D})$ over $\Spec K$ whose generic fiber $(\overline{\mathfrak{X}}\times_{\Spec O_{K}}\Spec K, \mathfrak{D}\times_{\Spec O_{K}}\Spec K)$ is isomorphic to $(\overline{X}, D)$ over $K$.
We refer to $(\overline{\mathfrak{X}}, \mathfrak{D})$ as a {\it smooth model} of $(\overline{X}, D)$.
\end{enumerate}
\label{reddef}
\end{dfn}

\begin{rem}
If a smooth model of a hyperbolic curve exists, it is unique up to canonical isomorphism by
\cite{DM} and \cite{Knu}.
\end{rem}

Let $(\overline{X}, D)$ be a hyperbolic curve $K$, $X$ the open subscheme $\overline{X} \setminus D$ of $\overline{X}$, and $ \overline{t}$ a geometric point of $X \times_{\mathrm{Spec}\,K} \Spec K^{\sep}$.
We have the following exact sequence of profinite groups:
\begin{eqnarray}
1 \rightarrow \pi _1 (X  \times_{\mathrm{Spec}\,K} \Spec K^{\sep} ,\overline{t} ) \rightarrow
 \pi _1 (X,\overline{t} ) \rightarrow G_K \rightarrow 1.
\label{3fhes}
\end{eqnarray} 
This exact sequence yields an outer Galois action
\begin{eqnarray}
G_{K} \rightarrow \mathrm{Out}(\pi _1 (X  \times_{\mathrm{Spec}\,K} \Spec K^{\sep},\overline{t})).
\end{eqnarray}
Then, for any prime number $l \neq p$, we have natural homomorphisms
\begin{eqnarray}
\begin{split}
I_{K} \hookrightarrow G_{K} &\rightarrow \mathrm{Out}(\pi _1 (X  \times_{\mathrm{Spec}\,K} \Spec K^{\sep},\overline{t})) \\
&\to \mathrm{Out}(\pi _1 (X  \times_{\mathrm{Spec}\,K} \Spec K^{\sep},\overline{t})^{p'})\\
&\to \mathrm{Out}(\pi _1 (X  \times_{\mathrm{Spec}\,K} \Spec K^{\sep},\overline{t})^{l}).
\end{split}\label{eq:3oa}
\end{eqnarray}
Oda and Tamagawa gave the following criterion:

\begin{prop}[\cite{Oda1}, \cite{Oda2}, and {\cite[Section 5]{Tama1}}]
The following are equivalent:
\begin{enumerate}
\item $(\overline{X}, D)$ has good reduction.
\item The outer action $I_K \rightarrow \mathrm{Out}( \pi _1 (X  \times_{\mathrm{Spec}\,K} \Spec K^{\sep} ,\overline{t}) ^{p'})$ defined by (\ref{eq:3oa}) is trivial.
\item There exists a prime number $l \neq p$ such that the outer action
$I_K \rightarrow \mathrm{Out}( \pi _1 (X  \times_{\mathrm{Spec}\,K} \Spec K^{\sep} ,\overline{t}) ^{l})$ defined by (\ref{eq:3oa}) is trivial.
\end{enumerate}
\label{3odatama}
\end{prop}

Suppose that $X$ is a proper hyperbolic curve over $K$ and there exists a section $s : \text{Spec}\,K \to X$.
Consider the geometric point $\overline{s}: \Spec K^{\sep}\to X\times_{\Spec K}\Spec K^{\sep}$ induced by $s$.
Then we have the exact sequence \eqref{3fhes} with $\overline{t}$ replaced by $\overline{s}$.
The section $s$ defines a section of the homomophism $ \pi _1 (X,\overline{s}) \rightarrow G_K $ in the homotopy exact sequence \eqref{3fhes}.
This induces a homomorphism $G_K \rightarrow \mathrm{Aut}( \pi _1 (X  \times_{\mathrm{Spec}\,K} \Spec K^{\sep} ,\overline{s}))$ such that the composite homomorphism $G_K \rightarrow \mathrm{Aut}( \pi _1 (X  \times_{\mathrm{Spec}\,K} \Spec K^{\sep} ,\overline{s})) \rightarrow \mathrm{Out}(\pi _1 (X  \times_{\mathrm{Spec}\,K} \Spec K^{\sep},\overline{s}))$ coincides with outer representation $G_{K} \rightarrow \mathrm{Out}(\pi _1 (X  \times_{\mathrm{Spec}\,K} \Spec K^{\sep},\overline{s}))$ in \eqref{eq:3oa}.
For a prime number $l \neq p$, we obtain homomorphisms
\begin{equation}
\begin{split}
I_{K} \hookrightarrow G_{K} &\rightarrow \mathrm{Aut}(\pi _1 (X  \times_{\mathrm{Spec}\,K} \Spec K^{\sep} ,\overline{s})) \\
&\rightarrow \mathrm{Aut}(\pi _1 (X  \times_{\mathrm{Spec}\,K} \Spec K^{\sep},\overline{s})^{p'}) \\
&\rightarrow \mathrm{Aut}(\pi _1 (X  \times_{\mathrm{Spec}\,K} \Spec K^{\sep},\overline{s})^{l}).
\end{split}
\label{3doa}
\end{equation}

\begin{prop}
The following are equivalent:
\begin{enumerate}
\item $X$ has good reduction.
\item The action of $I_K$ on $\pi _1 (X  \times_{\mathrm{Spec}\,K} \Spec K^{\sep} ,\overline{s} )^{p'}$ defined by (\ref{3doa}) is trivial.
\item The action of $I_K$ on $\pi _1 (X  \times_{\mathrm{Spec}\,K} \Spec K^{\sep} ,\overline{s} )^{l}$ defined by (\ref{3doa}) is trivial.
\end{enumerate}
\label{3gr}
\end{prop}

\begin{proof}
For the proof of the implication $1 \Rightarrow 2$, see Remark \ref{3rem}.
The implication $2 \Rightarrow 3$ is trivial.
Assume that the action of $I_{K}$ on $\pi _1 (X \times_{\Spec K} \Spec K^{\sep},\overline{s})^{l}$ is trivial.
Then the outer action $I_{K} \to \mathrm{Out}(\pi _1 (X \times_{\Spec K} \Spec K^{\sep},\overline{s})^{l})$ is trivial.
Therefore, $X$ has good reduction by Proposition \ref{3odatama}.
\end{proof}

\begin{rem}
In Proposition \ref{3odatama} and Proposition \ref{3gr}, to prove the implication $1 \Rightarrow 2$, we only need the hypothesis that the morphism $X \rightarrow \text{Spec}\,K$ is proper, smooth, and geometrically connected.
We show this assertion in the first part of the proof of Theorem \ref{main,thm}.1 and Theorem \ref{open,main,thm}.1.
\label{3rem}
\end{rem}

\label{repsect}

\section{First reduction}
In this section, we give a precise definition of a hyperbolic polycurve and the first step of the proof of Theorem \ref{main,thm}, \ref{open,main,thm}, and \ref{higher,main}.
Let $K, O_{K}, p, K^{\rm sep}, G_K,$ and $I_{K}$ be as in Section \ref{intro}.
Let $O^{\mathrm{h}}_{K}$ (resp.\,$O^{\mathrm{sh}}_{K}$) be the henseliazation (resp.\,strict henseliazation) of $O_{K}$ contained in $K^{\mathrm{sep}}$ defined by $I_{K}$. 

\begin{dfn}
Let $S$ be a scheme and $X$ a scheme over $S$.
\begin{enumerate}
\item
We shall say that $X$ is a {\it hyperbolic polycurve} (of relative dimension $n$) over $S$ if
there exists a (not necessarily unique) factorization of the structure morphism $X \rightarrow S$
\begin{equation}
\mathcal{S}: X = X_{n} \to X_{n-1} \to \ldots \to X_{1} \to X_{0} = S
\end{equation}
such that, for each $i \in \{ 1, \ldots ,n \}$, there exists a hyperbolic curve $(\overline{X}_{i}, D_{i})$ over $X_{i-1}$ (cf.\,Definition \ref{hyp.curve}) and the scheme $\overline{X}_{i} \setminus D_{i}$ is isomorphic to $X_{i}$ over $X_{i-1}$.
We refer to the above factorization of $X \rightarrow S$ as a {\it sequence of parameterizing morphisms}.
In the case where we consider a pair of a hyperbolic polycurve $X$ over $S$ and a sequence of parametrizing morphisms $\mathcal{S}$ of $X$, we write $(X, \mathcal{S})$.
We refer to such a pair as a {\it hyperbolic polycurve with a sequence of parametrizing morphisms}.
We shall say that two hyperbolic polycurves (over $S$) with a sequence of parametrizing morphisms $(X, \mathcal{S})$ and $(X', \mathcal{S'})$ are isomorphic if there exists an $S$-isomorphism between hyperbolic polycurves of relative dimension $i$ over $S$ defined by $\mathcal{S}$ and $\mathcal{S}'$ for each $1 \leq i \leq n$ such that these isomorphisms are compatible with the sequence of parametrizing morphisms $\mathcal{S}$ and $\mathcal{S}'$.
\item For a hyperbolic polycurve $X$ over $S$, the following are equivalent:
\begin{description} 
\item{(a)} The structure morphism $X \to S$ is proper.
\item{(b)} For any sequence of parameterizing morphisms
$$X = X_{n} \rightarrow X_{n-1} \rightarrow \ldots \rightarrow X_{1} \rightarrow X_{0} = S,$$ 
the morphism $X_{i} \rightarrow X_{i-1}$ is proper for each $1\leq i \leq n$.
\item{(c)} There exists a sequence of parameterizing morphisms
$$X = X_{n} \rightarrow X_{n-1} \rightarrow \ldots \rightarrow X_{1} \rightarrow X_{0} = S$$ 
such that the morphism $X_{i} \rightarrow X_{i-1}$ is proper for each $1\leq i \leq n$.
\end{description}
We call such $X \rightarrow S$ a {\it proper hyperbolic polycurve}.

\item Let $X$ be a hyperbolic polycurve (resp.\,proper hyperbolic polycurve) of relative dimension $n$ over $S$.
For a sequence of parameterizing morphisms
\begin{equation}
\mathcal{S} : X = X_{n} \rightarrow X_{n-1} \rightarrow \ldots \rightarrow X_{1} \rightarrow X_{0} = S,
\end{equation}
we write $b_{\mathcal{S}}$ (resp.\,$g_{\mathcal{S}}$) for the maximum of the first Betti numbers (resp.\,the genera) of fibers of all the morphisms $X_{i} \rightarrow X_{i-1}$ and refer to $b_{\mathcal{S}}$ (resp.\,$g_{\mathcal{S}}$)  {\it the maximal first Betti number} (resp.\,{\it the maximal genus}) of $\mathcal{S}$.
We write for $b_{X}$ (resp.\,$g_{X}$) the minimum of the maximal first Betti numbers (resp.\,the maximal genera) of sequences of parameterizing morphisms of $X$ and refer to $b_{X}$ (resp.\,$g_{X}$) as {\it the maximum first Betti number} (resp.\,{\it the maximum genus}) of $X$. 
\end{enumerate}
\label{hyperbolicpoly}
\end{dfn}

\begin{note-prop}
Let $1\to N \to G \to Q \to1$ be an exact sequence of profinite groups and $l\neq p$ a prime number.
Then $\mathrm{Ker}\,(N\to N^{p'})$ (resp.\,$\mathrm{Ker}\,(N\to N^{l})$) is a characteristic subgroup of $N$.
We write $G^{(p')}$ (resp.\,$G^{(l)}$) for the quotient group $G/\mathrm{Ker}\,(N\to N^{p'})$ (resp.\,$G/\mathrm{Ker}\,(N\to N^{l})$).
Moreover, we have the following commutative diagram with exact horizontal lines:
\begin{equation}
\xymatrix@=10pt{ 
1 \ar[r]
& N \ar[d] \ar[r] 
&G \ar[d] \ar[r]
&Q \ar[d] \ar[r]
& 1
&(\text{resp.}
&1 \ar[r]
& N \ar[d] \ar[r] 
&G \ar[d] \ar[r]
&Q \ar[d] \ar[r]
& 1\\
1 \ar[r]
&N^{p'} \ar[r]
&G^{(p')} \ar[r]
&Q\ar[r]
& 1
&
&1 \ar[r]
&N^{l} \ar[r]
&G^{(l)} \ar[r]
&Q\ar[r]
& 1).
}
\end{equation}
\label{prop'}
\end{note-prop}

\begin{proof}
Since every continuous homomorphism from $N$ to a pro-$p'$ profinite group factors $N\to N^{p'}$ (resp.\,$N\to N^{l}$), $\mathrm{Ker}\,(N\to N^{p'})$ (resp.\,$\mathrm{Ker}\,(N\to N^{l})$) is a characteristic subgroup of $N$.
\end{proof}

We start the proof of Theorem \ref{main,thm}, \ref{open,main,thm}, and \ref{higher,main}.

\begin{proof}[Proofs of Theorem \ref{main,thm}.1 and Therorem \ref{open,main,thm}.1 (cf.\,Remark \ref{3rem})]
Suppose that (A) holds.
Take a smooth model $\mathfrak{X}$ of $X$ (resp.\,the scheme $\mathfrak{X}$ in the condition (B)) if we are in the situation of Theorem \ref{main,thm}.1 (resp.\,Therorem \ref{open,main,thm}.2).
We have the following diagram:
\[
\xymatrix@=10pt{
X \times_{\Spec K}\Spec K^{\mathrm{sep}} \ar[d] \ar[r] & X\times_{\Spec K}\Spec(\mathrm{Frac}\,O^{\mathrm{h}}_{K}) \ar[d] \ar[r] 
& \mathrm{Spec}\,(\mathrm{Frac}\,O^{\mathrm{h}}_{K}) \ar[d] & \mathrm{Spec}\,K^{\mathrm{sep}} \ar[l] \ar[d] \\
\mathfrak{X}\times_{\Spec O_{K}}\Spec O^{\mathrm{sh}}_{K} \ar[r] & \mathfrak{X}\times_{\Spec O_{K}} \Spec O^{\mathrm{h}}_{K} \ar[r]
& \mathrm{Spec}\,O^{\mathrm{h}}_{K}
& \mathrm{Spec}\,O^{\mathrm{sh}}_{K}. \ar[l]
}
\]
Take a geometric point $\overline{\eta}$ of $X\times_{\Spec K}\Spec K^{\mathrm{sep}}$.
By Notation-Proposition \ref{prop'}, {\cite[Expos\'e IX, Theorem 6.1]{SGA1}}, and the same argument as the proof of  {\cite[Expos\'e IX, Theorem 6.1]{SGA1}}, we have the following commutative diagram of profinite groups with exact horizontal lines:
\begin{equation}
\xymatrix@=10pt{ 
1 \ar[r] & \pi _{1}(X\times_{\Spec K}\Spec K^{\mathrm{sep}},\overline{\eta})^{p'} \ar[d] \ar[r] 
& \pi _{1}(X\times_{\Spec K}\Spec(\mathrm{Frac}\,O^{\mathrm{h}}_{K}),\overline{\eta})^{(p')} \ar[d] \ar[r]
& \mathrm{Gal}(K^{\mathrm{sep}}/\mathrm{Frac}\,O^{\mathrm{h}}_{K}) \ar[d] \ar[r] & 1\\
1 \ar[r] & \pi _{1}(\mathfrak{X}\times_{\Spec O_{K}}\Spec O^{\mathrm{sh}}_{K},\overline{\eta})^{p'} \ar[r]
& \pi _{1}(\mathfrak{X}\times_{\Spec O_{K}}\Spec O^{\mathrm{h}}_{K},\overline{\eta})^{(p')} \ar[r]
& \pi _{1}(\mathrm{Spec}\,O^{\mathrm{h}}_{K},\overline{\eta}) \ar[r] & 1.
}
\label{syusei1}
\end{equation}
Since the first vertical arrow in the diagram (\ref{syusei1}) is an isomorphism by Theorem \ref{isom2} and the action $I_K$ on $\pi _{1}(\mathfrak{X}\times_{\Spec O_{K}}O^{\mathrm{sh}}_{K},\overline{\eta})^{p'}$ is trivial by the above diagram, the action of $I_K$ on $\pi _{1}(X\times_{\Spec K} K^{\mathrm{sep}},\overline{\eta})^{p'}$ is also trivial.
Hence, we finished the proof of Theorem \ref{main,thm}.1 and \ref{open,main,thm}.1.
\end{proof}

To prove Theorem \ref{higher,main} and the rest of the assertions in Theorem \ref{main,thm} and \ref{open,main,thm}, we need the following proposition:

\begin{prop}
Let $T$ be a regular integral separated scheme.
Let $Y$ be a scheme over $T$ satisfying the following condition:
There exists a factorization
$$Y = Y_{n} \rightarrow \ldots \rightarrow Y_{0} = \Spec O_{K}$$
such that there exist a proper smooth morphism $\overline{Y}_{i+1} \rightarrow Y_{i}$ with geometrically connected fibers and a normal crossing divisor $E_{i+1} \subset \overline{Y}_{i+1}$ of the scheme $\overline{Y}_{i+1}$ relative to $Y_{i}$ satisfying that the complement $\overline{Y}_{i+1} \setminus E_{i+1}$ is isomorphic to $Y_{i+1}$ for each $0 \leq i \leq n-1$.
Suppose that $T$ is a scheme over $\Z_{(p)}$ (resp.\,$\Q$) if $p> 0$ (resp.\,$p=0$).
Let $\eta = \mathrm{Spec}\,K(T)$ be the generic point of $T$, $K(T)^{\mathrm{sep}}$ a separable closure of $K(T)$, $\overline{\eta}$ the scheme $\Spec K(T)^{\mathrm{sep}}$, and $s$ a geometric point of $Y \times_{T} \overline{\eta}$.
By {\cite[Expos\'e IX, Theorem 6.1]{SGA1}} and {\cite[Theorem 0.2]{Nag3}}, we have the following commutative diagram with exact horizontal lines:

\[
\xymatrix{ 
1 \ar[r] & \pi_{1}(Y \times_{T} \overline{\eta}, s) \ar[d] \ar[r] 
& \pi_{1}(Y \times_{T} \eta, s) \ar[d] \ar[r]
& \pi_{1}(\eta, s) \ar[d] \ar[r] & 1\\
& \pi_{1}(Y \times_{T} \overline{\eta}, s) \ar[r]
& \pi_{1}(Y, s) \ar[r]
& \pi_{1}(T, s) \ar[r] & 1.
}
\]
Then the conjugate action of $\pi_{1}(Y \times_{T} \eta, s)$ on $\pi_{1}(Y \times_{T} \overline{\eta}, s)$ induces a natural action $\pi_{1}(Y, s) \rightarrow \mathrm{Aut}(\pi_{1}(Y \times_{T} \overline{\eta}, s)^{p'})$. Thus, we also obtain a natural outer action $\pi_{1}(T, s) \rightarrow \mathrm{Out}(\pi_{1}(Y \times_{T} \overline{\eta}, s)^{p'}).$
 
\label{outact}
\end{prop}

\begin{proof}

We need to show that the action $\pi_{1}(Y \times_{T} \eta, s) \rightarrow \mathrm{Aut}(\pi_{1}(Y \times_{T} \overline{\eta}, s)^{p'})$ factors through $\pi_{1}(Y \times_{T} \eta, s) \rightarrow \pi_{1}(Y, s).$

First, we show that we can reduce to the case where the scheme $T$ is the spectrum of a strictly henselian discrete valuation ring. Since the scheme $Y$ is regular, the kernel of the morphism $\pi_{1}(Y \times_{T} \eta, s) \rightarrow \pi_{1}(Y, s)$ is generated by inertia subgroups at points of codimension $1$ in $Y \setminus (Y \times_{T} \eta)$.
Consider such a point $y$ and write $t$ for the image of $y$ in $T$.
Since $Y$ is flat over $T$, $t$ is codimension $1$ in $T$. 
Choose a strict henselization $O_{T,t}^{\mathrm{sh}} \subset K(T)^{\mathrm{sep}}$ of the discrete valuation ring $O_{T,t}$ and write $\eta^{\mathrm{sh}}$ for the scheme $\mathrm{Spec}\,\mathrm{Frac}(O_{T,t}^{\mathrm{sh}})$.
Then the Galois group $\mathrm{Gal}(K(T)^{\mathrm{sep}} / \mathrm{Frac}(O_{T,t}^{\mathrm{sh}}))$ is an inertia subgroup of $t$ in $\pi_{1}(\eta, s)$.
Fix a separable closure $K(Y\times_{T} \overline{\eta})^{\mathrm{sep}}$ of the function field of $Y\times_{T} \overline{\eta}$.
Take a strict henselization $O_{Y,y}^{\mathrm{sh}}$ of the local ring $O_{Y,y}$ such that the diagram
\[
\xymatrix{ 
\mathrm{Spec}\,K(Y\times_{T} \overline{\eta})^{\mathrm{sep}} \ar[d] \ar[r] & \Spec O_{Y,y} \times_{T} \overline{\eta} \ar[d] \\
\mathrm{Spec}\,O_{Y,y}^{\mathrm{sh}} \ar[r] & \mathrm{Spec}\,O_{Y,y} \times_{T} \mathrm{Spec}\, O_{T,t}^{\mathrm{sh}}
}
\]
commutes.
Then the inertia subgroup of $y$ in $\pi_{1}(Y \times_{T} \eta, s)$ associated to $\mathrm{Spec}\,O_{Y,y}^{\mathrm{sh}}$ is sent into $\mathrm{Gal}(K(T)^{\mathrm{sep}} / \mathrm{Frac}(O_{T,t}^{\mathrm{sh}})) \subset \pi_{1}(\eta, s)$. 
Therefore, we have a commutative diagram

\[
\xymatrix{
1 \ar[r] & \pi_{1}(Y \times_{T} \overline{\eta}, s) \ar@{=}[d] \ar[r] 
& \pi_{1}(Y \times_{T} \eta^{\mathrm{sh}}, s) \ar@{^{(}-_>}[d] \ar[r]
& \pi_{1}(\eta^{\mathrm{sh}}, s) \ar@{^{(}-_>}[d] \ar[r] & 1\\
1 \ar[r] & \pi_{1}(Y \times_{T} \overline{\eta}, s) \ar[r] 
& \pi_{1}(Y \times_{T} \eta, s) \ar[r]
& \pi_{1}(\eta, s) \ar[r] & 1.
}
\]
The inertia subgroup of $y$ in $\pi_{1}(Y \times_{T} \eta, s)$ associated to $O_{Y,y}^{\mathrm{sh}}$ is contained in $\pi_{1}(Y \times_{T} \eta^{\mathrm{sh}}, s)$ and coincides with the inertia subgroup in $\pi_{1}(Y \times_{T} \eta^{\mathrm{sh}}, s)$ associated to the strict localization $\mathrm{Spec}\,O_{Y,y}^{\mathrm{sh}} \rightarrow Y\times_{T}\mathrm{Spec}\,O_{T,t}^{\mathrm{sh}}$.
Thus, to prove this proposition, it suffices to show that the homomorphism $\pi_{1}(Y\times_{T}\eta^{\mathrm{sh}},s) \rightarrow \mathrm{Aut}(\pi_{1}(Y\times_{T}\overline{\eta})^{p'})$ factors through $\pi_{1}(Y\times_{T}\eta^{\mathrm{sh}},s) \rightarrow \pi_{1}(Y\times_{T}\mathrm{Spec}\,O_{T,t}^{\mathrm{sh}},s)$.
Therefore, we can reduce \ref{outact} to the case where $T$ is the spectrum of a strictly henselian discrete valuation ring.

Suppose that $T$ is the spectrum of a strictly henselian discrete valuation ring.
We have a commutative diagram with an exact horizontal line
\[
\xymatrix{ 
1 \ar[r] & \pi_{1}(Y \times_{T} \overline{\eta}, s)^{p'} \ar[r] \ar[d] 
& \pi_{1}(Y \times_{T} \eta, s)^{(p')} \ar[ld] \ar[r]
& \pi_{1}(\eta, s) \ar[r] & 1\\
& \pi_{1}(Y, s)^{p'}.
}
\]
by Notation-Proposition \ref{prop'}.
Since the specialization homomorphism $\pi_{1}(Y \times_{T} \overline{\eta}, s)^{p'} \rightarrow \pi_{1}(Y, s)^{p'}$ is an isomorphism by Theorem \ref{isom2}, the assertions follow.
\end{proof}

We start the proof of Theorem \ref{higher,main} and the rest of the assertions in Theorem \ref{main,thm} and \ref{open,main,thm}.
When we are in the situation of Theorem \ref{main,thm}, fix a sequence of parameterizing morphisms
$$
X = X_{n} \overset{f_{n}}{\rightarrow} X_{n-1} \overset{f_{n-1}}{\rightarrow}  \ldots 
\overset{f_{i+1}}{\rightarrow} X_{i} \overset{f_{i}}{\rightarrow} X_{i-1} 
\overset{f_{i-1}}{\rightarrow} \ldots \overset{f_{2}}{\rightarrow} X_{1}
\overset{f_{1}}{\rightarrow} X_{0} = \mathrm{Spec}\,K
$$
of $X \rightarrow \mathrm{Spec}\,K$ whose maximal genus is $g_{X}$ and write $b$ for $2g_{X}$.
When we are in the situation of Theorem \ref{open,main,thm} or \ref{higher,main}, we write $b$ for $b_{\mathcal{S}}$.
Take a geometric point $\ast$ of the scheme $X \times_{\mathrm{Spec}\, K} \Spec K^{\mathrm{sep}}$.
Write $\Delta_{i}$ (resp.\,$\Pi_{i}$) for the \'etale fundamental group $\pi _1 (X_{i}  \times_{\mathrm{Spec}\, K} \Spec K^{\mathrm{sep}} ,\ast)$ (resp.\,$\pi _1 (X_{i} ,\ast)$).
By  {\cite[Expos\'e IX, Theorem 6.1]{SGA1}}, we have the following homotopy exact sequences of profinite groups
\begin{equation}
1 \to \Delta_{i} \to \Pi_{i} \to G_{K} \to 1.
\label{iexact}
\end{equation}

Assume that (B) or the conditions in Theorem \ref{main,thm}.4 holds.
Under the hypotheses given in Theorem \ref{main,thm} and Theorem \ref{open,main,thm}, we will show that (A) holds by induction on $n$.
For $n = 1$, this is proved in Proposition \ref{3odatama} and \ref{3gr}.
Hence, we assume that $n \geq 2$.

First, we prove that (A) holds for $X_{n-1}$.
Since the morphism $X_{n-1} \rightarrow \mathrm{Spec}\,K$ has a structure of a hyperbolic polycurve associated with the above sequence of parameterizing morphisms, we can apply the main theorems by induction hypothesis if we show that the outer Galois representation
\begin{equation}
I_{K} \rightarrow \mathrm{Out}(\Delta_{n-1}^{p'})
\label{induction}
\end{equation}
is trivial. 
Since $X_{n} \rightarrow X_{n-1}$ has geometrically connected fibers, the homomorphism $\Delta_{n} \rightarrow \Delta_{n-1}$ is surjective.
Then we can show that the outer Galois representation (\ref{induction}) is trivial by using the commutative diagram of profinite groups with exact horizontal lines
\[
\xymatrix{ 
1 \ar[r] & \Delta_{n} \ar@{->>}[d] \ar[r] 
& \Pi_{n} \ar@{->>}[d] \ar[r]
& G_{K} \ar@{=}[d] \ar[r] & 1\\
1 \ar[r] & \Delta_{n-1} \ar[r]
& \Pi_{n-1} \ar[r]
& G_{K} \ar[r] & 1.
}
\]

When we are in the situation of Theorem \ref{main,thm}, we may assume that there exists a smooth model $\mathfrak{X}_{n-1}$ of $X_{n-1}$.
When we are in the situation of Theorem \ref{open,main,thm}, we may assume that there exists a hyperbolic polycurve
\begin{equation}
\mathfrak{X}_{n-1} \rightarrow \ldots \rightarrow \mathfrak{X}_{1} \rightarrow \mathfrak{X}_{0} = \mathrm{Spec}\,O_{K}
\end{equation}
whose generic fiber is isomorphic to the sequence of parameterizing morphisms
\begin{equation}
X_{n-1} \overset{f_{n-1}}{\rightarrow} \ldots 
\overset{f_2}{\rightarrow} X_1 
\overset{f_1}{\rightarrow} \mathrm{Spec}\,K .
\end{equation}
Consider the following diagram:
\[
\xymatrix{
X_{n}\times_{X_{n-1}} \mathrm{Spec}\,K(X_{n-1}) \ar[r] \ar[d]^{f'_{n}} & X_{n} \ar[d]^{f_{n}} \ar@{.>}[r]
&  \ar@{.>}[d]
& \ar@{.>}[l] \ar@{.>}[d] \\
\mathrm{Spec}\,K(X_{n-1}) \ar[r] & X_{n-1} \ar[r]
& \mathfrak{X}_{n-1}
& \mathrm{Spec}\,O_{\mathfrak{X}_{n-1},\xi} , \ar[l]
}
\]
where $K(X_{n-1})$ is the function field of $X_{n-1}$, $\xi$ is the generic point of the special fiber
$\mathfrak{X}_{n-1} \setminus  X_{n-1}$, $f'_{n}$ is the base change of $f_{n}$, and $O_{\mathfrak{X}_{n-1},\xi}$ is the local ring of $\mathfrak{X}_{n-1}$ at $\xi$.
It suffices to show that there exists a hyperbolic curve $(\overline{\mathfrak{X}}_{n},\mathfrak{D_{n}})$ over $\mathfrak{X}_{n-1}$ such that the scheme $(\overline{\mathfrak{X}}_{n}\setminus\mathfrak{D}_{n})\times_{\mathfrak{X}_{n-1}}X_{n-1}$ is isomorphic to $X_{n}$ over $X_{n-1}$.
By Proposition \ref{mor} (cf.\,\cite{Mor}), it suffices to show that there exists a hyperbolic curve $(\overline{\mathfrak{X}}, \mathfrak{D})$ over $\mathrm{Spec}\,O_{\mathfrak{X}_{n-1},\xi}$ such that the scheme $(\overline{\mathfrak{X}} \setminus \mathfrak{D})\times_{\mathrm{Spec}\,O_{\mathfrak{X}_{n-1},\xi}} \mathrm{Spec}\,K(X_{n-1})$ is isomorphic to $X_{n}\times_{X_{n-1}} \mathrm{Spec}\,K(X_{n-1})$ over $\mathrm{Spec}\,K(X_{n-1})$.
Let $I_{\xi} \subset G_{K(X_{n-1})}$ be an inertia subgroup $\xi$ and $\overline{t}$ a geometric point of $X_{n}\times _{X_{n-1}} \mathrm{Spec}\,K(X_{n-1})$.
Write $\Delta_{n,n-1}$ for the \'etale fundamental group $\pi _{1}(X_{n}\times _{X_{n-1}}\overline{t},\overline{t})$.
To complete the proof, it suffices to show that the outer representation of $I_{\xi}$ on $\Delta_{n,n-1}^{l}$ is trivial for some prime number $l \neq p$ by Proposition \ref{3odatama}.
Note that the homomorphism
$$I_{\xi} \rightarrow \mathrm{Out}\,\Delta_{n,n-1}^{l}$$
is the composite of the natural morphisms
\begin{equation}
I_{\xi} \rightarrow G_{K(X_{n-1})} \rightarrow \pi _{1}(X_{n-1},\overline{t})
\label{inertia}
\end{equation}
and the outer representation $\pi _{1}(X_{n-1},\overline{t}) \rightarrow \mathrm{Out}\,\Delta_{n,n-1}^{l}$ constructed in Proposition \ref{outact}.
Write $I$ for the image of the inertia subgroup $I_{\xi}$ in (\ref{inertia}) to $\pi _{1}(X_{n-1},\overline{t})$, which coincides with an inertia subgroup of $\pi _{1}(X_{n-1},\overline{t})$ at $\xi$.
Therefore, we have
$$\mathrm{Ker}(\Pi_{n-1} \rightarrow \pi_{1}(\mathfrak{X}_{n-1}, \overline{t})) = (\text{the subgroup topologically normally generated by} \, I).$$
Then we have the following proposition in summary:

\begin{prop} 
To prove Theorem \ref{main,thm}.2, 3, 4 and Theorem \ref{open,main,thm}.2, 3, it suffices to show that the outer action
$I \rightarrow \mathrm{Out}\,\Delta_{n,n-1}^{l}$
is trivial for some prime number $l \neq p$.
\label{claim}
\end{prop}

By the argument as in the proof of Proposition \ref{outact}, the image of $I$ by the homomorphism $\Pi_{n-1} \rightarrow G_{K}$ is contained in some inertia subgroup of $G_{K}$.
Therefore, we may assume that $K$ is strictly henselian.

\label{reduction}

\section{The case of residual characteristic $0$}
We prove Theorem \ref{main,thm}.2 and Theorem \ref{open,main,thm}.2 in this section.
Hence, we may assume that $p=0$ and $K$ is strictly henselian.
By Proposition \ref{claim}, it suffices to show that the outer action $I \rightarrow \mathrm{Out}\,{\Delta_{n,n-1}}$ is trivial.

Since the characteristic of the field $K$ is also $0$, we have homotopy exact sequences
$$1 \to \Delta_{n,n-1} \to \Pi_{n} \to \Pi_{n-1} \rightarrow 1$$
and
\begin{equation}
1 \rightarrow \Delta_{n,n-1} \rightarrow \Delta_{n} \rightarrow \Delta_{n-1} \rightarrow 1
\label{extension}
\end{equation}
by {\cite[PROPOSITION 2.5]{Ho}}.
Note that $\Delta_{n}$ is center-free, which follows from {\cite[Proposition 1.11]{Tama1}}, the homotopy exact sequence (\ref{extension}), and the fact that an extension group of center-free groups is also center-free.
For a group $G$ and a subgroup $H$ of $G$, we write $Z_{G}(H)$ for the centralizer subgroup of $H$ in $G$.

\begin{lem}
There exist decompositions of profinite groups
$$\Pi_{n} = \Delta_{n} \times Z_{\Pi_{n}}(\Delta_{n})$$
and
$$\Pi_{n-1} = \Delta_{n-1} \times Z_{\Pi_{n-1}}(\Delta_{n-1}),$$
which are compatible with the natural surjection $\Pi_{n} \rightarrow \Pi_{n-1}.$
\label{lem1}
\end{lem}

\begin{proof}
Since $\Delta_{n}$ is center-free, we have
$$ \Delta_{n} \cap Z_{\Pi_{n}}(\Delta_{n}) = \{ 1 \}.$$
Moreover, since $K$ is strictly henselian and the outer action
$$I_{K} (\cong \Pi_{n}/\Delta_{n}) \rightarrow \mathrm{Out}\,(\Delta_{n})$$
is trivial, we have a canonical decomposition $\Pi_{n} = \Delta_{n} \times Z_{\Pi_{n}}(\Delta_{n}).$
By the same argument, we obtain a canonical decomposition $\Pi_{n-1} = \Delta_{n-1} \times Z_{\Pi_{n-1}}(\Delta_{n-1})$.
The homomorphism $\Pi_{n} \rightarrow \Pi_{n-1}$ is compatible with the homomorphism $\Delta_{n} \rightarrow \Delta_{n-1}$, and hence $Z_{\Pi_{n}}(\Delta_{n}) \rightarrow Z_{\Pi_{n}}(\Delta_{n-1})$ is well-defined. Therefore, these decomposition of groups are compatible.
\end{proof}

\begin{lem}
$$I \subset Z_{\Pi_{n-1}}(\Delta_{n-1}).$$
\label{lem2}
\end{lem}
\begin{proof}
We have the following commutative diagram with exact horizontal line:
\[
\xymatrix{ 
& & I \ar@{^{(}-_>}[d]
& &\\
1 \ar[r] & \Delta_{n-1} \ar[r] \ar[rd]
& \Pi_{n-1} \ar[d] \ar[r]
& G_{K} \ar[r] & 1\\
& & \pi_{1}(\mathfrak{X}_{n-1},\overline{t}).
& &
}
\]
Here, the homomorphism $\Delta_{n-1} \rightarrow \pi_{1}(\mathfrak{X}_{n-1},\overline{t})$ is an isomorphism by Theorem \ref{isom2}.
Then we obtain a decomposition
\begin{equation}
\Pi_{n-1} \cong \Delta_{n-1} \times \mathrm{Ker}(\Pi_{n-1} \rightarrow \pi_{1}(\mathfrak{X}_{n-1}, \overline{t})).
\label{finalout}
\end{equation}
Therefore, $I \subset \mathrm{Ker}(\Pi_{n-1} \rightarrow \pi_{1}(\mathfrak{X}_{n-1}, \overline{t})) \subset Z_{\Pi_{n-1}}(\Delta_{n-1})$ holds.
\end{proof}
Now, we prove Theorem \ref{main,thm}.2 and Theorem \ref{open,main,thm}.2.
\begin{proof}[Proofs of Theorem \ref{main,thm}.2 and Theorem \ref{open,main,thm}.2]
By Lemma \ref{lem2}, it suffices to show that the outer action
$$Z_{\Pi_{n-1}}(\Delta_{n-1}) \rightarrow \mathrm{Out}\,{\Delta_{n,n-1}}$$
is trivial.
By Lemma \ref{lem1}, the homomorphism $Z_{\Pi_{n}}(\Delta_{n}) \rightarrow Z_{\Pi_{n-1}}(\Delta_{n-1})$ is surjective.
Therefore, the outer action (\ref{finalout}) is trivial.
We finished the proof of Theorem \ref{main,thm}.2 and Theorem \ref{open,main,thm}.2.
\end{proof}
\label{(0,0)}

\section{The case of residual characteristic $p > 0$}
We maintain the notation of Section \ref{reduction}.
In this section, we will prove Theorem \ref{main,thm}.3, 4 and Theorem \ref{open,main,thm}.3.
Hence, we assume that $p > b+1$.
It suffices to prove that the outer action $I \rightarrow \mathrm{Out}\,\Delta_{n,n-1}^{l}$
is trivial for some prime number $l \neq p$ by Proposition \ref{claim}.
Let us take a prime number $l$ which is a generator of the cyclic group $(\Z/p\Z)^{*}$, whose existence follows from the theorem on arithmetic progression.
Since $l, l^{2}, \ldots , l^{b}$ are not $1 \in (\Z/p\Z)^{*}$ by the hypothesis $p > b+1$, the order of the group $\mathrm{GL}(b, \F_{l})$ is not divisible by $p$.
Also, the profinite group $\mathrm{Ker}(\mathrm{Aut}(\Delta_{n,n-1}^{l}) \rightarrow \mathrm{Aut}((\Delta_{n,n-1}^{l})^{\mathrm{ab}}/l(\Delta_{n,n-1}^{l})^{\mathrm{ab}}))$ is pro-$l$ by a well known theorem of P.Hall.
Here, the superscript ``ab'' denotes the abelianization of the profinite group.
Therefore, the profinite groups $\mathrm{Aut}(\Delta_{n,n-1}^{l})$ and $\mathrm{Out}(\Delta_{n,n-1}^{l})$ are pro-prime-to-$p$.

Note that we have two exact sequences of profinite groups
$$\Delta_{n,n-1} \rightarrow \Pi_{n} \rightarrow \Pi_{n-1} \rightarrow 1$$
and
$$\Delta_{n,n-1} \rightarrow \Delta_{n} \rightarrow \Delta_{n-1} \rightarrow 1$$
by {\cite[PROPOSITION 1.10]{Ho}}.
Unfortunately, if the characteristic of $K$ is equal to $p$, the homomorphism $\Delta_{n,n-1} \rightarrow \Delta_{n}$ is not always injective (cf.\,{\cite[Theorem (0.3)]{Tama2}}).

We have the following commutative diagram with exact horizontal lines:

\[
\xymatrix{ 
& \Delta_{n,n-1} \ar[d] \ar[r] 
& \Delta_{n} \ar[d] \ar[r]
& \Delta_{n-1} \ar[d] \ar[r] & 1\\
1 \ar[r] & \mathrm{Inn}(\Delta_{n,n-1}^{l}) \ar[r]
& \mathrm{Aut}(\Delta_{n,n-1}^{l}) \ar[r]
& \mathrm{Out}(\Delta_{n,n-1}^{l}) \ar[r] & 1.
}
\]
Here, the homomorphism $\Delta_{n} \rightarrow \mathrm{Aut}(\Delta_{n,n-1}^{l})$ is constructed in Proposition \ref{outact}. Since the profinite group $\mathrm{Out}(\Delta_{n,n-1}^{l})$ is a pro-prime-to-$p$ group, the outer action $\Delta_{n-1} \rightarrow \mathrm{Out}(\Delta_{n,n-1}^{l})$ factors through $\Delta_{n-1}^{p'}$.
We write $\Delta_{n}^{(l,p')}$ for the pull-back $\mathrm{Aut}(\Delta_{n,n-1}^{l}) \times_{\mathrm{Out}(\Delta_{n,n-1}^{l})} \Delta_{n-1}^{p'}$.
Then we have the following commutative diagram with exact horizontal lines:

\[
\xymatrix{ 
& \Delta_{n,n-1} \ar[d] \ar[r] 
& \Delta_{n} \ar[d] \ar[r]
& \Delta_{n-1} \ar[d] \ar[r] & 1\\
1 \ar[r] & \mathrm{Inn}(\Delta_{n,n-1}^{l}) \ar[r]  \ar[d]
& \Delta_{n}^{(l,p')} \ar[r] \ar[d]
& \Delta_{n-1}^{p'}\ar[r] \ar[d] & 1\\
1 \ar[r] & \mathrm{Inn}(\Delta_{n,n-1}^{l}) \ar[r]
& \mathrm{Aut}(\Delta_{n,n-1}^{l}) \ar[r]
& \mathrm{Out}(\Delta_{n,n-1}^{l}) \ar[r] & 1.
}
\]
Since the profinite group $\Delta_{n,n-1}^{l}$ is center-free, the homomorphism
$\Delta_{n,n-1}^{l} \rightarrow \mathrm{Inn}(\Delta_{n,n-1}^{l})$ is an isomorphism.
Therefore, we obtain an exact sequence
\begin{equation}
1 \rightarrow \Delta_{n,n-1}^{l} \rightarrow \Delta_{n}^{(l,p')} \rightarrow \Delta_{n-1}^{p'} \rightarrow 1.
\label{seq1}
\end{equation}

Next, consider the following diagram of exact sequences with exact horizontal lines:
\[
\xymatrix{ 
& \Delta_{n,n-1} \ar[d] \ar[r] 
& \Pi_{n} \ar[d] \ar[r]
& \Pi_{n-1} \ar[d] \ar[r] & 1\\
1 \ar[r] & \mathrm{Inn}(\Delta_{n,n-1}^{l}) \ar[r]
& \mathrm{Aut}(\Delta_{n,n-1}^{l}) \ar[r]
& \mathrm{Out}(\Delta_{n,n-1}^{l}) \ar[r] & 1.
}
\]
Here, the homomorphism $\Pi_{n} \rightarrow \mathrm{Aut}(\Delta_{n,n-1}^{l})$ is constructed in Proposition \ref{outact}.
By applying Notation-Proposition \ref{prop'} to the exact sequences (\ref{iexact}), we have an exact sequence of profinite groups
\begin{equation}
1 \to \Delta_{n-1}^{p'} \to \Pi_{n-1}^{(p')} \to G_{K} \to1.
\label{seq2}
\end{equation}
We write $\Pi_{n}^{((l,p'))}$ for the pull-back $\mathrm{Aut}(\Delta_{n,n-1}^{l}) \times_{\mathrm{Out}(\Delta_{n,n-1}^{l})} \Pi_{n-1}^{(p')}$.
Then we have the following commutative diagram with exact horizontal lines:
\[
\xymatrix{ 
& \Delta_{n,n-1} \ar[d] \ar[r] 
& \Pi_{n} \ar[d] \ar[r]
& \Pi_{n-1} \ar[d] \ar[r] & 1\\
1 \ar[r] & \mathrm{Inn}(\Delta_{n,n-1}^{l}) \ar[r]  \ar[d]
&\Pi_{n}^{((l,p'))} \ar[r] \ar[d]
& \Pi_{n-1}^{(p')} \ar[r] \ar[d] & 1\\
1 \ar[r] & \mathrm{Inn}(\Delta_{n,n-1}^{l}) \ar[r]
& \mathrm{Aut}(\Delta_{n,n-1}^{l}) \ar[r]
& \mathrm{Out}(\Delta_{n,n-1}^{l}) \ar[r] & 1.
}
\]
Therefore, we obtain an exact sequence
\begin{equation}
1 \rightarrow \Delta_{n,n-1}^{l} \rightarrow \Pi_{n}^{((l,p'))} \rightarrow \Pi_{n-1}^{(p')} \rightarrow 1.
\label{seq3}
\end{equation}
From the three exact sequences (\ref{seq1}), (\ref{seq2}), and (\ref{seq3}), we have an exact sequence
$$1 \rightarrow \Delta_{n}^{(l,p')} \rightarrow \Pi_{n}^{((l,p'))} \rightarrow G_{K} \rightarrow1.$$

\begin{proof}[Proofs of Theorem \ref{main,thm}.3 and Theorem \ref{open,main,thm}.3]
If $n = 2$, the profinite groups $\Delta_{2,1}^{l}$ and $\Delta_{1}^{p'}$ are center-free. Therefore, the profinite group $\Delta_{2}^{(l,p')}$ is also center-free. Since Lemma \ref{lem1} and Lemma \ref{lem2} work if we replace 
$\Pi_{n}, \Delta_{n}, \Pi_{n-1}$ and $\Delta_{n-1}$ by $\Pi_{2}^{((l,p'))}, \Delta_{2}^{(l,p')}, \Pi_{1}^{(p')}$ and $\Delta_{1}^{p'}$, we can prove that the outer action $I \rightarrow \mathrm{Out}\,\Delta_{2,1}^{l}$ is trivial.
Thus, we finished the proofs of Theorem \ref{main,thm}.3 and Theorem \ref{open,main,thm}.3.
\end{proof}

\begin{proof}[Proof of Theorem \ref{main,thm}.4]
Suppose that $X$ has a $K$-rational point $x$ and that the Galois representation 
$I_{K(x)} \rightarrow \mathrm{Aut}(\pi _1 (X  \times_{\Spec K} \Spec K^{\mathrm{sep}} ,\overline{x})^{p'})$ defined as in (\ref{3doa}) in Section \ref{repsect} is trivial.
Take a path from the fundamental functor defined by $\overline{x}\to X\times_{\Spec K}\Spec K^{\sep}$ to that defined by $\overline{t}\to X\times_{\Spec K}\Spec K^{\sep}$.
Then we have an induced homomorphism $I_{K(x)}\to \Pi_{n}$ and the induced action of $I_{K(x)}$ on $\Delta_{n}^{p'}$ is trivial.
We have a commutative diagram with exact horizontal lines
\[
\xymatrix{ 
& & I_{K(x)} \ar[d] & I \ar[d]\\
& \Delta_{n,n-1} \ar[d] \ar[r] 
& \Pi_{n} \ar[d] \ar[r]
& \Pi_{n-1} \ar[d] \ar[r] & 1\\
1 \ar[r] & \Delta_{n,n-1}^{l} \ar[r]  \ar@{^{(}-_>}[d]
&\Pi_{n}^{((l,p'))} \ar[r] \ar[d]
& \Pi_{n-1}^{(p')}\ar[r] \ar[d] & 1\\
1 \ar[r] & \Delta_{n}^{(l,p')} \ar[r]
&\Pi_{n}^{((l,p'))} \ar[r]
& G_{K} \ar[r] & 1.
}
\]
Since the action of $I_{K(x)}$ on $\Delta_{n}^{p'}$ is trivial, the action of $I_{K(x)}$ on $\Delta_{n}^{(l,p')}$ is also trivial.
Hence, the action of $I_{K(x)}$ on $\Delta_{n,n-1}^{l}$ is trivial.
Since the image of $I$ in $\Pi_{n-1}^{(p')}$ is contained in $\mathrm{Ker}(\Pi_{n-1}^{(p')} \to \pi_{1}(\mathfrak{X}_{n-1},\overline{t})^{p'})$, it suffice to show that the image of $I_{K(x)}$ in $\Pi_{n-1}^{(p')}$ coincides with $\mathrm{Ker}(\Pi_{n-1}^{(p')} \to \pi_{1}(\mathfrak{X}_{n-1},\overline{t})^{p'})$ to show that the outer action $I\to \mathrm{Out}\Delta_{n,n-1}^{l}$ is trivial.

By valuative criterion, the composite morphism $\Spec K(x) \to X_{n-1} \to \mathfrak{X}_{n-1}$ factors 
the morphism $\Spec K(x) \to \Spec O_{K(x)}$.
Therefore, we have a natural homomorphism $(I_{K(x)} =)\, G_{K(x)} \to \mathrm{Ker}(\Pi_{n-1}^{(p')} \to \pi_{1}(\mathfrak{X}_{n-1},\overline{t})^{p'})$.
Since the composite morphism $\Delta_{n-1}^{p'} \to\Pi_{n-1}^{(p')} \to  \pi_{1}(\mathfrak{X}_{n-1},\overline{t})^{p'}$ is an isomorphism by Theorem \ref{isom2}, the composite homomorphism $\mathrm{Ker}(\Pi_{n-1}^{(p')} \to \pi_{1}(\mathfrak{X}_{n-1},\overline{t})^{p'})\hookrightarrow \Pi_{n-1}^{(p')}\to G_{K}$ is an isomorphism by the exactness of the sequence (\ref{seq2}).
Therefore, the homomorphism $(I_{K(x)} =)\, G_{K(x)} \to \mathrm{Ker}(\Pi_{n-1}^{(p')} \to \pi_{1}(\mathfrak{X}_{n-1},\overline{t})^{p'})$ is an isomorphism.
Thus, we finished the proof of Theorem \ref{main,thm}.4.
\end{proof}

\label{pcase}

\section{The case of residual characteristic $p \gg 0$}
In this section, we prove Theorem \ref{higher,main}.
To prove Theorem 1.3, we need a very strong condition on $p$.
In this section, we maintain the notation of Section \ref{pcase} and suppose that the outer Galois representation $I_{K}\to\mathrm{Out}\,\Delta_{X}^{l}$ is trivial.

\begin{lem}[cf.\,{\cite[Lemma2.18]{Saw}} and {\cite[Proposition 3]{A}}]

Let
\begin{equation}
\xymatrix{ 
1 \ar[r] & N \ar[r] \ar@{=}[d]
& \tilde{G} \ar[d] \ar[r]
& \tilde{H} \ar[d] \ar[r] & 1\\
& N \ar[r] & G \ar[r] & H \ar[r] & 1
}
\label{AndSaw}
\end{equation}
be a commutative diagram of profinite groups with exact horizontal lines such that the middle and the right vertical lines are surjective.
Suppose that the conjugate action $\tilde{G} \rightarrow \mathrm{Aut}(N^{l})$ admits a factorization
\begin{equation}
\tilde{G} \rightarrow G \rightarrow \mathrm{Aut}(N^{l}).
\label{conjugate}
\end{equation}
Moreover, suppose that $N^{l}$ is topologically finitely generated and center-free.
Then the following are equivalent:
\begin{enumerate}
\item
The pro-$l$ completion of the sequence
$$1 \rightarrow N^{l} \rightarrow G^{l} \rightarrow H^{l} \rightarrow 1$$
induced by the lower line of (\ref{AndSaw}) is exact.
\item
The induced homomorphism $N^{l} \rightarrow G^{l}$ is injective.
\item The action $G \rightarrow \mathrm{Aut}(N^{l})$ given in (\ref{conjugate}) factors through $G \rightarrow G^{l}$.
\item The outer action $H \rightarrow \mathrm{Out}(N^{l})$ induced by (\ref{conjugate}) factors through $H \rightarrow H^{l}$.
\item The action $H \rightarrow \mathrm{Aut}(N^{\mathrm{ab},l}/lN^{\mathrm{ab},l})$ induced by (\ref{conjugate}) factors through $H \rightarrow H^{l}$.
\end{enumerate}
\label{Ander}
\end{lem}

\begin{proof}
The implications 1 $\Rightarrow$ 2 $\Rightarrow$ 3 $\Rightarrow$ 4 $\Rightarrow$ 5 are trivial.
Since $N^{l}$ is topologically finitely generated, the profinite group $\mathrm{Ker}(\mathrm{Aut}(N^{l}) \rightarrow \mathrm{Aut}((N^{\mathrm{ab}}/l(N^{l})^{\mathrm{ab}}))$ is pro-$l$ by a well known theorem of P.Hall.
Therefore, the implication 5 $\Rightarrow$ 4 holds.
The image of the homomorphism $G \rightarrow \mathrm{Aut}(N^{l})$ is pro-$l$ if and only if the image of the homomorphism $H \rightarrow \mathrm{Out}(N^{l})$ is pro-$l$.
Therefore, the implication 4 $\Rightarrow$ 3 holds.
Since $N^{l}$ is center-free, the homomorphism $N^{l} \to \mathrm{Inn}(N^{l})$ is isomorphic, and hence the composite homomorphism $N^{l} \to \mathrm{Inn}\to \mathrm{Aut}(N^{l})$ is injective.
Therefore, the implication 3 $\Rightarrow$ 2 holds.
Since taking pro-$l$ completion is a right exact functor,  the implication 2 $\Rightarrow$ 1 holds.
\end{proof}

\begin{prop}
If there exist a quotient group
$\Delta^{(l,p')}_{n} \rightarrow \overline{\Delta}_{n}$ and a center-free quotient group $\Delta_{n-1}^{p'} \rightarrow \overline{\Delta}_{n-1}$
such that the diagram
\[
\xymatrix{ 
1 \ar[r] & \Delta_{n,n-1}^{l} \ar[r]  \ar[d]
& \Delta_{n}^{(l,p')} \ar[r] \ar[d]
& \Delta_{n-1}^{p'}\ar[r] \ar[d] & 1\\
1 \ar[r] & \Delta_{n,n-1}^{l} \ar[r]
& \overline{\Delta}_{n} \ar[r]
& \overline{\Delta}_{n-1} \ar[r] & 1
}
\]
commutes and the second horizontal line is exact, then the outer action $I \rightarrow \mathrm{Out}(\Delta_{n,n-1}^{l})$ is trivial.
\label{bargp}
\end{prop}

\begin{proof}
Since the outer actions $I_{K} \rightarrow \mathrm{Out}(\Delta_{n}^{p'})$ and $I_{K} \rightarrow \mathrm{Out}(\Delta_{n-1}^{p'})$ are trivial, we have $\Pi_{n}^{((l,p'))} = \Delta_{n}^{(l,p')}Z_{\Pi_{n}^{((l,p'))}}(\Delta_{n}^{(l,p')})$ and $\Pi_{n-1}^{(p')} = \Delta_{n-1}^{p'} Z_{\Pi_{n-1}^{(p')}}(\Delta_{n-1}^{p'})$.
Therefore, the normal subgroup $K_{n}\deq\mathrm{Ker}(\Delta_{n}^{(l,p')} \to \overline{\Delta}_{n})$ (resp.\,$K_{n-1}\deq\mathrm{Ker}(\Delta_{n-1}^{p'} \to \overline{\Delta}_{n-1})$) of $\Delta^{(l,p')}_{n}$ (resp.\,$\Delta^{p'}_{n-1}$) is also a normal subgroup of $\Pi_{n}^{((l,p'))}$ (resp. $\Pi_{n-1}^{(p')}$).
Thus, we have quotient groups
$\overline{\Pi}_{n} \deq \Pi_{n}^{((l,p'))}/K_{n}$ and $\overline{\Pi}_{n-1} \deq \Pi_{n-1}^{(p')}/K_{n-1}$ such that the diagram
\[
\xymatrix{ 
1 \ar[r] & \Delta_{n,n-1}^{l} \ar[r] \ar[d]
& \overline{\Delta}_{n} \ar[r] \ar[d]
& \overline{\Delta}_{n-1} \ar[r] \ar[d] & 1\\
1 \ar[r] & \Delta_{n,n-1}^{l} \ar[r]
& \overline{\Pi}_{n} \ar[r]
& \overline{\Pi}_{n-1} \ar[r] & 1
}
\]
commutes.
Since the group $\overline{\Delta}_{n-1}$ is center-free, we have a decomposition $\overline{\Pi}_{n-1} = \overline{\Delta}_{n-1} \times Z_{\overline{\Pi}_{n-1}}(\overline{\Delta}_{n-1} )$.
If we replace $\Delta_{n-1}$, $\Pi_{n-1}$, and $\pi_{1}(\mathfrak{X}_{n-1}, \overline{t})$ in Lemma \ref{lem2} by $\overline{\Delta}_{n-1}$, $\overline{\Pi}_{n-1}$, and $\pi_{1}(\mathfrak{X}_{n-1}, \overline{t})/\mathrm{Im}(K_{n-1}\to\pi_{1}(\mathfrak{X}_{n-1}, \overline{t}))$, it follows that the image of the group $I$ in $\overline{\Pi}_{n-1}$ is contained in $Z_{\overline{\Pi}_{n-1}}(\overline{\Delta}_{n-1})$.
Since the homomorphism $\overline{\Pi}_{n} \rightarrow \overline{\Pi}_{n-1}$ is surjective and compatible with the homomorphism $\overline{\Delta}_{n} \rightarrow \overline{\Delta}_{n-1}$, the image of $Z_{\overline{\Pi}_{n}}(\overline{\Delta}_{n})$ in $\overline{\Pi}_{n-1}$ coincides with $Z_{\overline{\Pi}_{n-1}}(\overline{\Delta}_{n-1})$.
Thus, the outer action $I \rightarrow \mathrm{Out}\,{\Delta}_{n,n-1}^{l}$, which factors through $Z_{\overline{\Pi}_{n-1}}(\overline{\Delta}_{n-1}) \rightarrow \mathrm{Out}\,{\Delta}_{n,n-1}^{l}$, is trivial (cf.\,the proof in the end of Section \ref{(0,0)}).
\end{proof}

\begin{lem}
Assume that $n \geq 3$. Let $\Delta_{n-1}'$ be an open normal subgroup of $\Delta_{n-1}$.
Write $\Delta_{i}'$ (resp.\,$\Delta_{n}'$) for the images of $\Delta_{n-1}'$ in $\Delta_{i}$ for $1 \leq i \leq n-2$ (resp.\,the inverse image of $\Delta_{n-1}'$ in $\Delta_{n}$), $\Delta_{i,i-1}'$ for the inverse images of $\Delta_{i}'$ in $\Delta_{i,i-1}$ for $2 \leq i \leq n-1$, $\Gamma_{i}$ (resp.\,$\Gamma_{i,i-1}$) for the group $\Delta_{i}/\Delta'_{i}$ (resp.\,$\Delta_{i,i-1}/\Delta'_{i,i-1}$) for $1 \leq i \leq n$ (resp.\,$2 \leq i \leq n-1$), $\Delta_{i}^{(l)}$ (resp.\,$\Delta_{i,i-1}^{(l)}$) for the quotient group of $\Delta_{i}$ (resp.\,$\Delta_{i,i-1}$) defined by Notation-Proposition \ref{prop'} and the exact sequence $1\to \Delta'_{i} \to \Delta_{i} \to\Gamma_{i}\to1$ (resp.\,$1\to \Delta'_{i,i-1} \to \Delta_{i,i-1} \to\Gamma_{i,i-1} \to 1$ ) for $1 \leq i \leq n$ (resp.\,$2 \leq i \leq n-1$).
Suppose that the sequence $1 \to (\Delta_{i,i-1}')^{l} \to (\Delta_{i}')^{l} \to (\Delta_{i-1}')^{l} \to 1$
is exact for each $2 \leq i \leq n-1$.
\begin{enumerate}
\item
The profinite group $(\Delta_{i}')^{l}$ is center-free for each $2 \leq i \leq n-1$.
\item
Write $\tilde{X}_{i}'$ for the Galois covering of $\tilde{X}_{i} \deq X_{i}\times_{\mathrm{Spec}\,K} \mathrm{Spec}\,K^{\mathrm{sep}}$ corresponding to the normal subgroup $\Delta_{i}'$ of $\Delta_{i}$.
Then the composite homomorphism
$\Gamma_{i} (= \mathrm{Aut}(\tilde{X}_{i}'/\tilde{X}_{i})) \hookrightarrow \mathrm{Aut}(\tilde{X}_{i}'/\mathrm{Spec}\, K^{\mathrm{sep}}) \rightarrow \mathrm{Out}((\Delta_{i}')^{l})$ is injective for each $1\leq i \leq n-1$.
\item
Suppose moreover that the composite homomorphism
$$\Delta_{n-1}' \hookrightarrow \Delta_{n-1} \to \mathrm{Out}(\Delta_{n,n-1}^{l}) \rightarrow \mathrm{Aut}(\Delta_{n,n-1}^{l,\mathrm{ab}}/l\Delta_{n,n-1}^{l,\mathrm{ab}}),$$
where $\Delta_{n-1} \to \mathrm{Out}(\Delta_{n,n-1}^{l} )$ is constructed in Proposition \ref{outact}, is trivial and $\Gamma_{n-1}$ is of order prime-to-$p$. 
Then the group $\Delta_{n-1}^{(l)}$ is center-free and we have the following commutative diagram with exact horizontal lines:
$
\xymatrix{ 
1 \ar[r] & \Delta_{n,n-1}^{l} \ar[r]  \ar[d]
& \Delta_{n}^{(l,p')} \ar[r] \ar[d]
& \Delta_{n-1}^{p'}\ar[r] \ar[d] & 1\\
1 \ar[r] & \Delta_{n,n-1}^{l} \ar[r]
& \Delta_{n}^{(l)} \ar[r] 
& \Delta_{n-1}^{(l)} \ar[r] & 1.
}
$
\end{enumerate}
\label{injective}
\end{lem}

\begin{proof}
Assertion 1 follows from {\cite[Proposition 1.11]{Tama1}} and the fact that an extension group of center-free groups is also center-free.

Next, We prove assertion 2.
Note that, if the characteristic of $K$ is $0$, this is a special case of {\cite[Proposition 3.2]{Saw}}.
We will show the injectivity of the outer action $\Gamma_{i} \to \mathrm{Out}((\Delta_{i}')^{l})$ by induction on $i$.
To show that the outer action $\Gamma_{1}\to  \mathrm{Out}((\Delta_{1}')^{l})$ is injective, it suffices to show that the action $\Gamma_{1}\to \mathrm{Aut}((\Delta_{1}')^{l,\mathrm{ab}})$ is injective.
Let $K^{\mathrm{alg}}$ be an algebraic closure of $K^{\mathrm{sep}}$ and write $Y$ (resp.\,$Y'$) for the scheme $\tilde{X}_{1} \times_{\mathrm{Spec}\,K^{\mathrm{sep}}}\mathrm{Spec}\,K^{\mathrm{alg}}$ (resp.\,$\tilde{X}'_{1} \times_{\mathrm{Spec}\,K^{\mathrm{sep}}}\mathrm{Spec}\,K^{\mathrm{alg}}$).
There exists a hyperbolic curve $(\overline{Y'},E')$ over $K^{\mathrm{alg}}$ such that $\overline{Y'}\setminus E'$ is isomorphic to $Y'$ over $K^{\mathrm{alg}}$.
Such a pair always exists since we work over $K^{\mathrm{alg}}$.
Let $\phi$ be an element of the group $\Gamma_{1}(= \mathrm{Aut}(\tilde{X}_{1}'/\tilde{X}_{1}) = \mathrm{Aut}(Y'/Y))$ whose image in the group
$$\mathrm{Aut}((\Delta_{1}')^{l,\mathrm{ab}}) (= \mathrm{Aut}((\pi_{1}(\tilde{X}_{1}',\ast)^{l,\mathrm{ab}}) = \mathrm{Aut}((\pi_{1}(Y',\ast)^{l,\mathrm{ab}}))$$
is trivial.
Here, $\ast$ is a geometric point of $Y'$.
Then $\phi$ induces the identity on $E'$.
Therefore, if the genus of $\overline{Y'}$ is equal to or less than $1$ (resp.\,equal to or more than $2$), $\phi$ is the identity by the theory of automorphisms of rational curves and elliptic curves (resp.\,by {\cite[Theorem 1.13]{DM}}).

Suppose that assertion 2 holds for each $1\leq i\leq n-2$.
Write $\mathrm{Aut}((\Delta_{n-1}')^{l},(\Delta_{n-2}')^{l})$ for the subgroup of $\mathrm{Aut}((\Delta_{n-1}')^{l})$ consisting of automorphisms of $(\Delta_{n-1}')^{l}$ inducing automorphisms of the quotient group $(\Delta_{n-2}')^{l}$ and $\mathrm{Out}((\Delta_{n-1}')^{l},(\Delta_{n-2}')^{l})$ for the quotient group of $\mathrm{Aut}((\Delta_{n-1}')^{l},(\Delta_{n-2}')^{l})$ by the inner subgroup $\mathrm{Inn}((\Delta_{n-1}')^{l})$.
 Then we have the following diagram:
 \[
\xymatrix{ 
\Gamma_{n-1} \ar[r] \ar[d]
& \mathrm{Out}((\Delta_{n-1}')^{l},(\Delta_{n-2}')^{l}) \ar@{^{(}-_>}[r] \ar[d]
& \mathrm{Out}((\Delta_{n-1}')^{l})\\
\Gamma_{n-2}\ar[r]
& \mathrm{Out}((\Delta_{n-2}')^{l}),
&
}
\]
where the homomorphism $\Gamma_{n-2} \to \mathrm{Out}((\Delta_{n-2}')^{l})$ is injective by induction hypothesis.
Write $\zeta'$ for the spectrum of a separable closure of the function field of $\tilde{X}_{n-2}'$, $\tilde{X}_{n-1,n-2}$ for the scheme $\tilde{X}_{n-1} \times_{\tilde{X}_{n-2}} \zeta'$, and $\tilde{X}_{n-1,n-2}'$ for the scheme $\tilde{X}_{n-1}' \times_{\tilde{X}_{n-2}'} \zeta'$.
Then the Galois covering $\tilde{X}_{n-1,n-2}' \rightarrow \tilde{X}_{n-1,n-2}$ corresponds to the normal subgroup $\Delta_{n-1,n-2}' \subset \Delta_{n-1,n-2}$. Therefore, we obtain the following diagram:
 \[
\xymatrix@=10pt{
& \mathrm{Aut}(\tilde{X}_{n-1,n-2}'/\tilde{X}_{n-1,n-2}) \ar@{^{(}-_>}[r] \ar@{=}[d] & \mathrm{Aut}(\tilde{X}_{n-1}'/\tilde{X}_{n-1}) \ar[r] \ar@{=}[d] & \mathrm{Out}((\Delta_{n-1}')^{l})\\
1 \ar[r] & \Gamma_{n-1,n-2} \ar[r] & \Gamma_{n-1} \ar[r] &\Gamma_{n-2},
}
\]
where the lower horizontal line is exact.
Thus, it suffices to show that the composite homomorphism $\Gamma_{n-1,n-2} \to \Gamma_{n-1} \to \mathrm{Out}((\Delta_{n-1}')^{l})$ is injective. 

Consider the following diagram with exact horizontal lines and vertical lines:
\begin{equation}
\xymatrix@=10pt{
& 1\ar[d] & 1\ar[d] & 1\ar[d] & \\
1 \ar[r] & (\Delta_{n-1,n-2}')^{l}  \ar[r] \ar[d] & (\Delta_{n-1}')^{l}  \ar[r] \ar[d] & (\Delta_{n-2}')^{l} \ar[r] \ar[d] & 1 \\
1 \ar[r] & (\Delta_{n-1,n-2})^{(l)}  \ar[r] \ar[d] & (\Delta_{n-1})^{(l)}  \ar[r] \ar[d] & (\Delta_{n-2})^{(l)} \ar[r] \ar[d] & 1\\
1 \ar[r]
&\Gamma_{n-1,n-2} \ar[r] \ar[d] 
&\Gamma_{n-1} \ar[r] \ar[d]
&\Gamma_{,n-2}\ar[r] \ar[d]
& 1\\
& 1 & 1 & 1. &
}
\label{nnewdiagram}
\end{equation}
Here, the second horizontal line in the diagram (\ref{nnewdiagram}) is exact by nine lemma.
Since the homomorphism $\Gamma_{n-1,n-2}\to \mathrm{Out}((\Delta_{n-1,n-2}')^{l})$ is injective by the argument of the first step of the induction, it suffices to show that an element $g$ of $(\Delta_{n-1,n-2})^{(l)}$, whose action on $(\Delta_{n-1}')^{l}$ is same as the inner action of some $h \in (\Delta_{n-1}')^{l}$, induces an inner action on $(\Delta_{n-1,n-2}')^{l}$.
The image of $g$ in $(\Delta_{n-2})^{(l)}$ is trivial, which induces a trivial action on $(\Delta_{n-2}')^{l}$.
Since $(\Delta_{n-2}')^{l}$ is center-free, $h$ is in $(\Delta_{n-1,n-2}')^{l}$.
Hence, we finished the proof of assertion 2.

Finally, we prove assertion 3.
We have a diagram with exact horizontal lines
\[
\xymatrix{ 
\Delta_{n,n-1} \ar[r]  \ar@{=}[d]
& \Delta_{n}' \ar[r] \ar[d]
& \Delta_{n-1}' \ar[r] \ar[d] & 1\\
\Delta_{n,n-1} \ar[r] 
& \Delta_{n} \ar[r] \ar[d]
& \Delta_{n-1} \ar[r] \ar[d] & 1\\
& \Gamma_{n-1} \ar@{=}[r]
& \Gamma_{n-1},
}
\]
which induces a commutative diagram
\begin{equation}
\xymatrix{ 
1 \ar[r] & \Delta_{n,n-1}^{l} \ar[r]  \ar[d]
& (\Delta_{n}')^{l} \ar[r] \ar[d]
& (\Delta_{n-1}')^{l} \ar[r] \ar[d] & 1\\
1 \ar[r] & \Delta_{n,n-1}^{l} \ar[r]
& \Delta_{n}^{(l)} \ar[r] \ar[d]
& \Delta_{n-1}^{(l)} \ar[r] \ar[d] & 1\\
& & \Gamma_{n-1} \ar@{=}[r]
& \Gamma_{n-1}.
}
\label{use Lem6.1}
\end{equation}
Since the homomorphism $\Delta'_{n-1} \rightarrow \mathrm{Aut}(\Delta_{n,n-1}^{l,\mathrm{ab}}/l\Delta_{n,n-1}^{l,\mathrm{ab}})$ is trivial and the group $\Delta_{n,n-1}^{l}$ is center-free, the first horizontal line of the diagram (\ref{use Lem6.1}) is exact by Lemma \ref{Ander}.
(Note that we can construct the diagram of the form (\ref{AndSaw}) as in the statement of Proposition \ref{outact} with $N=\Delta_{n-1,n},G=\Delta_{n}'$, and $H=\Delta_{n-1}'$.)
Hence, we see the exactness of the second line by diagram chasing.
Since $(\Delta_{n-1}')^{l}$ is center-free by assertion 1 and the outer action $\Gamma_{n-1} \to\mathrm{Out}((\Delta_{n-1}')^{l})$ is injective by assertion 2, the group $\Delta_{n-1}^{(l)}$ is center-free.
The order of $\Gamma_{n-1}$ is prime to $p$, and hence $\Delta_{n-1}^{(l)}$ is a pro-$p'$ group. Therefore, we obtain the surjective homomorphisms $\Delta_{n-1}^{p'} \rightarrow \Delta_{n-1}^{(l)}$ and $\Delta_{n}^{(l,p')} \rightarrow \Delta_{n}^{(l)}$.
Thus, we have the desired commutative diagram with exact horizontal lines
\[
\xymatrix{ 
1 \ar[r] & \Delta_{n,n-1}^{l} \ar[r]  \ar[d]
& \Delta_{n}^{(l,p')} \ar[r] \ar[d]
& \Delta_{n-1}^{p'}\ar[r] \ar[d] & 1\\
1 \ar[r] & \Delta_{n,n-1}^{l} \ar[r]
& \Delta_{n}^{(l)} \ar[r] 
& \Delta_{n-1}^{(l)} \ar[r] & 1.
}
\]
\end{proof}

Now we prove Theorem \ref{higher,main}.

\begin{prop}
Assume that $n \geq 3$ and $p \neq 2$.
Let $l \neq p$ be a prime number.
Define a function $f_{b,l}(m)$ for $m \geq 3$ in the following way:
\begin{itemize}
\item
For $m=3$, $f_{b,l}(3) = l^{b^{2}}$.
\item
For $m \geq 3$,
$$f_{b,l}(m+1)= (f_{b,l}(m)) \times (l^{b^{2} \times f_{b,l}(m)^{2}})^{f_{b,l}(m)}.$$
\end{itemize}
If $p > l^{b \times f_{b,l}(n)}$, (B) of Theorem \ref{higher,main} implies (A) of Theorem \ref{higher,main}.
\label{higher}
\end{prop}

\begin{proof}
If we can find an open normal subgroup $\Delta_{n-1}'$ of $\Delta_{n-1}$ satisfying the assumptions of Lemma \ref{injective}.3, the groups $\Delta_{n}^{(l)}$ and $\Delta_{n-1}^{(l)}$ in Lemma \ref{injective}.3 work as $\overline{\Delta}_{n}$ and $\overline{\Delta}_{n-1}$ in Proposition \ref{bargp}, and hence (A) holds  by Proposition \ref{claim}.

Let us construct $\Delta_{n-1}'$ in the following.
Fix $2 \leq i \leq n-1$ and assume that there exists an open normal subgroup $\Delta_{n-1}^{i}$ of  $\Delta_{n-1}$  such that the images $ \Delta_{j}^{i}$ of $\Delta_{n-1}^{i}$ in $\Delta_{j}$  for $i \leq j \leq n-1$ and the inverse images $ \Delta_{j,j-1}^{i}$ of $ \Delta_{j}^{i}$ in $\Delta_{j,j-1}$ for $i \leq j \leq n-1$ induce exact sequences
$$1 \to ( \Delta_{j,j-1}^{i})^{l} \to ( \Delta_{j}^{i})^{l} \to ( \Delta_{j-1}^{i})^{l} \rightarrow 1$$ 
for all $i+1 \leq j \leq n-1$. 
Write $\tilde{\Delta}_{i-1}^{i}$ for the image of $ \Delta_{n-1}^{i}$ in $\Delta_{i-1}$.
By Proposition \ref{outact}, the exact sequence
$$\Delta_{i,i-1}^{i} \rightarrow  \Delta_{i}^{i} \rightarrow  \tilde{\Delta}_{i-1}^{i} \rightarrow 1$$
induces homomorphisms $\tilde{\Delta}_{i-1}^{i} \to \mathrm{Out}((\Delta_{i,i-1}^{i})^{l})$.
Write $\alpha$ for the composite homomorphism $\tilde{\Delta}_{i-1}^{i} \to \mathrm{Out}((\Delta_{i,i-1}^{i})^{l}) \to \mathrm{Aut}((\Delta_{i,i-1}^{i})^{l,\mathrm{ab}}/l(\Delta_{i,i-1}^{i})^{l,\mathrm{ab}})$,
$\Delta_{i-1}^{i-1}$ for the maximum normal subgroup of $\Delta_{i-1}$ contained in $\mathrm{Ker}\,\alpha$, $\Delta_{n-1}^{i-1}$ for the inverse image of $\Delta_{i-1}^{i-1}$ in $\Delta_{n-1}^{i}$, $\Delta_{j}^{i-1}$ for the image of $\Delta_{n-1}^{i-1}$ in $\Delta_{j}^{i}$ for each $i \leq j \leq n$, and $\Delta_{j,j-1}^{i-1}$ for the inverse image of $\Delta_{j}^{i-1}$ in $\Delta_{j,j-1}^{i}$ for each $i \leq j \leq n$.
Note that $\Delta_{j}^{i-1}$ coincides with the inverse image of $\Delta_{i-1}^{i-1}$ in $\Delta_{j}^{i}$ for any $i \leq j \leq n$, $\Delta_{j,j-1}^{i} = \Delta_{j,j-1}^{i-1}$ for any $i \leq j \leq n$, and $\Delta_{j}^{i-1}$ is a normal subgroup of $\Delta_{j}$ for any $i \leq j \leq n-1$.
Thus, we have exact sequences of profinite groups
$$ \Delta_{j,j-1}^{i} \rightarrow  \Delta_{j}^{i-1} \rightarrow  \Delta_{j-1}^{i-1} \rightarrow 1$$
which induce the exact sequences
$$1 \rightarrow ( \Delta_{j,j-1}^{i})^{l} \rightarrow ( \Delta_{j}^{i-1})^{l} \rightarrow ( \Delta_{j-1}^{i-1})^{l} \rightarrow 1$$
for all $i \leq j \leq n-1$ by the same argument of the proof of the exactness of the first line of the diagram (\ref{use Lem6.1}).

Starting from $\Delta_{n-1}^{n-1} = \mathrm{Ker}(\Delta_{n-1} \rightarrow \mathrm{Aut}((\Delta_{n,n-1}^{n-1})^{l,\mathrm{ab}}/l(\Delta_{n,n-1}^{n-1})^{l,\mathrm{ab}})$, we reach the construction of $\Delta_{j}^{1}$ for $1 \leq j \leq n-1$.
Then we define $\Delta_{n-1}'$ to be $\Delta_{n-1}^{1}$.
To finish the proof of Proposition \ref{higher}, it suffices to show that the quotient group $\Gamma_{n-1} = \Delta_{n-1}/ \Delta_{n-1}^{1}$ is of order prime to $p$.

We recall the structure of the group $\Gamma_{n-1} = \Delta_{n-1}/ \Delta_{n-1}^{1}$ and show the order of $\Gamma_{n-1}$ is prime to $p$.
Since we have
\begin{equation*}
\begin{split}
\Delta_{n-1}^{1} &= \underset{1 \leq i \leq n-1}{\bigcap} \Delta_{n-1}^{i} \\
&=  \underset{1 \leq i \leq n-2}{\bigcap}(\text{the inverse image of}\, \Delta_{i}^{i}\, \text{in}\, \Delta_{n-1}),
\end{split}
\end{equation*}
it suffices to show that the order of the group $\Delta_{i}/\Delta_{i}^{i}$ is prime-to-$p$ for any $1 \leq i \leq n-2$. 

\begin{claim}
For each $1 \leq j \leq n-1$, the order of the group $\Delta_{n-j}/\Delta_{n-j}^{n-j}$ is prime-to-$p$ and $\leq f_{b,l}(j+2)$.
\label{newclaim}
\end{claim}

We show Claim \ref{newclaim} by induction on $j$.
First, we consider the case $j = 1$.
It holds that
\begin{equation*}
\mathrm{dim}_{\F_{l} }((\Delta_{n,n-1})^{l,\mathrm{ab}}/l(\Delta_{n,n-1})^{l,\mathrm{ab}}\leq b
\end{equation*}
and
\begin{equation*}
\begin{split}
&\#\mathrm{Aut}((\Delta_{n,n-1})^{l,\mathrm{ab}}/l(\Delta_{n,n-1})^{l,\mathrm{ab}})\\
=\,& \underset{0 \leq j \leq \mathrm{dim}_{\F_{l} }(\Delta_{n,n-1})^{l,\mathrm{ab}}/l(\Delta_{n,n-1})^{l,\mathrm{ab}}-1}\prod (l^{\mathrm{dim}_{\F_{l} }(\Delta_{n,n-1})^{l,\mathrm{ab}}/l(\Delta_{n,n-1})^{l,\mathrm{ab}}}-l^{j})\\
\mid \,& \underset{0 \leq j \leq b-1}\prod (l^{b}-l^{j})\\
= \,& l^{(b-1)b /2} \times \underset{1 \leq j \leq b}\prod (l^{j}-1)\\
 \leq \,& l^{(b-1)b /2} \times l^{b(b+1)/2} = l^{b^{2}} = f_{b,l}(3).
\end{split}
\end{equation*}
Here, the notation $a \mid c$ means that $a$ divides $c$.
Since $l^{b} < p$, the group
$$\Delta_{n-1}/\Delta_{n-1}^{n-1} (\hookrightarrow \mathrm{Aut}((\Delta_{n,n-1})^{l,\mathrm{ab}}/l(\Delta_{n,n-1})^{l,\mathrm{ab}}))$$
is of order prime-to-$p$ and $\leq f_{b,l}(3)$.

Fix $1\leq j \leq n-2$ and assume that Claim 6.5 holds for each $i$ satisfying that $1 \leq i \leq j$.
We show that the order of the group $\Delta_{n-j-1}/\Delta_{n-j-1}^{n-j-1}$ is prime-to-$p$ and $\leq f_{b,l}(j+3)$.
We have a surjection
$$\Delta_{n-j}/\Delta_{n-j}^{n-j}  \rightarrow \Delta_{n-j-1}/\tilde{\Delta}_{n-j-1}^{n-j},$$
which shows that the order of $\Delta_{n-j-1}/\tilde{\Delta}_{n-j-1}^{n-j}$ is prime-to-$p$.
Recall that the group $\Delta_{n-j-1}^{n-j-1}$ is the maximum normal subgroup of $\Delta_{n-j-1}$ contained in $\mathrm{Ker}(\tilde{\Delta}_{n-j-1}^{n-j} \rightarrow \mathrm{Aut}(( \Delta_{n-j,n-j-1}^{n-j})^{l,\mathrm{ab}}/l( \Delta_{n-j,n-j-1}^{n-j})^{l,\mathrm{ab}}))$.
To see that the group $\Delta_{n-j-1}/\Delta_{n-j-1}^{n-j-1}$ is of order prime-to-$p$, 
it suffices to show that the image of the homomorphism
$$\tilde{\Delta}_{n-j-1}^{n-j} \rightarrow \mathrm{Aut}((\Delta_{n-j,n-j-1}^{n-j})^{l,\mathrm{ab}}/l(\Delta_{n-j,n-j-1}^{n-j})^{l,\mathrm{ab}})$$
is of order prime-to-$p$.
We have the estimates of the dimension of the $\F_{l}$-linear space

\begin{equation*}
\begin{split}
\mathrm{dim}_{\F_{l} }(\Delta_{n-j,n-j-1}^{n-j})^{l,\mathrm{ab}}/l(\Delta_{n-j,n-j-1}^{n-j})^{l,\mathrm{ab}} & \leq b \times (\Delta_{n-j,n-j-1} : \Delta_{n-j,n-j-1}^{n-j})\\
& \leq b \times (\Delta_{n-j} : \Delta_{n-j}^{n-j})\\
& = b \times f_{b,l}(j+2)
\end{split}
\end{equation*}

and the order of the group

\begin{equation*}
\begin{split}
&\#\mathrm{Aut}((\Delta_{n-j,n-j-1}^{n-j})^{l,\mathrm{ab}}/l(\Delta_{n-j,n-j-1}^{n-j})^{l,\mathrm{ab}})\\
\mid \, & \underset{0 \leq i \leq b \times (\Delta_{n-j,n-j-1} : \Delta_{n-j,n-j-1}^{n-j})-1}\prod (l^{b \times (\Delta_{n-j,n-j-1} : \Delta_{n-j,n-j-1}^{n-j})}-l^{i})\\
\mid \, & \underset{0 \leq i \leq b \times (\Delta_{n-j} : \Delta_{n-j}^{n-j})-1}\prod (l^{b \times (\Delta_{n-j} : \Delta_{n-j}^{n-j})}-l^{i})\\
\mid \, & \underset{0 \leq i \leq b \times f_{b,l}(j+2) - 1}\prod (l^{b \times f_{b,l}(j+2)}-l^{i})\\
= \, & l^{(b \times f_{b,l}(j+2)-1)(b \times f_{b,l}(j+2))/2} \times \underset{1 \leq i \leq b \times f_{b,l}(j+2)}\prod (l^{i}-1)\\
\leq \, & l^{b^{2} \times f_{b,l}(j+2)^{2}} .
\end{split}
\end{equation*}
Since $l^{b \times f_{b,l}(j+2)} \leq l^{b \times f_{b,l}(n)} < p$, the group $\tilde{\Delta}_{n-j-1}^{n-j}/\Delta_{n-j-1}^{n-j-1}$ is of order prime-to-$p$, and hence $\Delta_{n-j-1}/\Delta_{n-j-1}^{n-j-1}$ is also of order prime-to-$p$.
The desired estimate of the order of $\Delta_{n-j-1}/\Delta_{n-j-1}^{n-j-1}$ is obtained in the following way:
\allowdisplaybreaks[4]
\begin{align*}
& (\Delta_{n-j-1} : \Delta_{n-j-1}^{n-j-1}) \notag \\
\leq  \, & (\Delta_{n-j-1} : \tilde{\Delta}_{n-j-1}^{n-j})\notag\\
& \times \#\mathrm{Im}(\tilde{\Delta}_{n-j-1}^{n-j} \rightarrow \mathrm{Aut}(( \Delta_{i,i-1}^{i})^{l,\mathrm{ab}}/l( \Delta_{i,i-1}^{i})^{l,\mathrm{ab}}))^{(\Delta_{n-j-1}:\tilde{\Delta}_{n-j-1}^{n-j})}\notag\\
\leq  \, & (\Delta_{n-j-1}:\tilde{\Delta}_{n-j-1}^{n-j}) \times (\#\mathrm{Aut}((\Delta_{n-j,n-j-1}^{n-j})^{l,\mathrm{ab}}/l(\Delta_{n-j,n-j-1}^{n-j})^{l,\mathrm{ab}}))^{(\Delta_{n-j-1} : \tilde{\Delta}_{n-j-1}^{n-j})}\notag\\
\leq \, &  (\Delta_{n-j-1}:\tilde{\Delta}_{n-j-1}^{n-j}) \times (l^{b^{2} \times f_{b,l}(j+2)^{2}})^{(\Delta_{n-j-1} : \tilde{\Delta}_{n-j-1}^{n-j})}\notag\\
\leq \, & (\Delta_{n-j}:\Delta_{n-j}^{n-j}) \times (l^{b^{2} \times f_{b,l}(j+2)^{2}})^{(\Delta_{n-j} : \Delta_{n-j}^{n-j})}\notag\\
\leq \, & f_{b,l}(j+2) \times (l^{b^{2} \times f_{b,l}(j+2)^{2}})^{f_{b,l}(j+2)} = f_{b,l}(j+3).
\end{align*}
Therefore, Claim \ref{newclaim} holds.
\end{proof}

\section{Appendix 1: Notes on extension of smooth curves}
In this section, we review the extension property of family of proper hyperbolic curves proved in \cite{Mor} and prove a non-proper version of this property (cf.\,Remark \ref{Stix}).
Hoshi informed the author of the proof of Proposition \ref{mor}.
\begin{prop}[cf.\,\cite{Mor}]
Let $S$ be a connected regular Noetherian scheme, $U$ an open subscheme of $S$ such that the codimension of $S \setminus U$ in $S$ is $\geq 2$, $(\overline{X_{U}},D_{U}) \to U$ a hyperbolic curve, and $g$ the genus of $\overline{X_{U}}$.
Then there exists a hyperbolic curve $(\overline{X},D) \to S$ such that $(\overline{X}\times_{S}U,D\times_{S}U)$ is isomorphic to $(\overline{X_{U}},D_{U})$ over $U$, which is unique up to unique isomorphism.
\label{mor}
\end{prop}

\begin{proof}
The uniqueness follows from the separatedness of the moduli stacks of hyperbolic curves (cf.\,\cite{DM} and \cite{Knu}).

If $D_{U}$ is zero, the assertion follows from {\cite[Th\'eor\`eme (ii)]{Mor}}.
Suppose that the \'etale morphism $D_{U} \rightarrow U$ is of degree $r \geq 1$.
By Galois descent, we may assume that the scheme $D_{U}$ is the disjoint union $\underset{1\leq i \leq r}{\coprod} D_{i,U}$ of $r$ copies of $S$.
Let us assume that the extension $\overline{X} \rightarrow S$ and $D_{i} \rightarrow \overline{X}$ of $\overline{X_{U}} \rightarrow U$ and $D_{i,U} \rightarrow \overline{X_{U}}$ exist.
Then $D_{i}$ and $D_{j}$ are disjoint if $i \neq j$.
Otherwise, if we denote the diagonal of $X \times_{S} X$ by $\Delta_{X/S}$, $\Delta_{X/S} \cap (D_{i} \times_{S} D_{j} )$ is codimension $1$ in $D_{i} \times_{S} D_{j} \cong S$, which is a contradiction.
Thus, it suffices to show that the extension $\overline{X} \rightarrow S$ (resp.\, $D_{i} \rightarrow \overline{X}$) of $\overline{X_{U}} \rightarrow U$ (resp.\, $D_{i,U} \rightarrow \overline{X_{U}}$) exists.

If $g \geq 1$, we have an extension $\overline{X} \rightarrow S$ of $\overline{X_{U}} \rightarrow U$ by {\cite[Th\'eor\`eme (ii)]{Mor}}.
From {\cite[Lemme1]{Mor}}, we have an extension $D_{i} \rightarrow \overline{X} \rightarrow S$ of the section $D_{i,U} \rightarrow \overline{X_{U}} \rightarrow U$ for each $i$.

If $g = 0$, there exists an isomorphism between $\overline{X_{U}}$ and $\mathbb{P}^{1}_{U}$ over $U$ such that the sections $D_{1,U}, D_{2,U}$, and $D_{3,U}$ correspond to $0$,$1$, and $\infty$.
Therefore, we consider the extension $\overline{X} = \mathbb{P}^{1}_{S} \rightarrow S$ and we will show that extensions $D_{i} \rightarrow S$ of $D_{i,U} \rightarrow U$ exist for all $i \geq 4$.
To give the divisors $D_{i,U} \rightarrow \mathbb{P}^{1}_{U}$ for $i \geq 4$ is equivalent to give a section $s_{i}$ of $\Gamma(U,O_{U})$ other than $0$, $1$, and to give their extensions $D_{i} \rightarrow \mathbb{P}^{1}_{S}$ is equivalent to give extensions of $s_{i}$ to elements of $\Gamma(S,O_{S})$.
Since $S \setminus U$ is of codimension $\geq 2$ in the normal scheme $S$, we have $\Gamma(S,O_{S}) \cong \Gamma(U,O_{U})$, and hence the morphisms $D_{i,U} \rightarrow \mathbb{P}^{1}_{U}$ extend to $D_{i} \rightarrow \mathbb{P}^{1}_{S}$.
\end{proof}

\begin{rem}
It is mentioned in {\cite[Remarks 2.11 (d)]{Stix}} that Lemma \ref{mor} is already shown in \cite{Mor}, but it seems not to be shown.
As mentioned there, Lemma \ref{mor} is equivalent to {\cite[Theorem 1.2]{Stix}} by Zariski-Nagata purity theorem, under the hypothesis of Lemma \ref{mor}.
Here, we gave an elementary and direct proof.
\label{Stix}
\end{rem}

\section{Appendix 2: An example of the fundamental group of a hyperbolic polycurve}

In this section, we give an example of a hyperbolic polycurve with bad reduction such that the pro-$l$ outer Galois representation of the inertia subgroup is trivial for all but one prime $l$.

\begin{lem}
Let $1 \rightarrow N \rightarrow G \rightarrow H \rightarrow 1$ be an exact sequence of profinite groups.
\begin{enumerate}
\item We have an exact sequence
$$(N/[N,\mathrm{Ker}(G \rightarrow G^{l})])^{l} \rightarrow G^{l} \rightarrow H^{l} \rightarrow 1.$$
Here, $[-,-]$ denotes the closure of the commutator subgroup.
\item
Suppose that we have a section $s$ of the homomorphism $G \rightarrow H$.
Write $N_{\mathrm{Ker}(H \rightarrow H^{l})}$ for the maximal quotient group of $N$ on which $\mathrm{Ker}(H \rightarrow H^{l})$ acts trivially.
Then we have an exact sequence
$$(N_{\mathrm{Ker}(H \rightarrow H^{l})})^{l} \rightarrow G^{l} \rightarrow H^{l} \rightarrow 1.$$
\end{enumerate}
\label{pro-l-exact}
\end{lem}

\begin{proof}
Since the image of $[N,\mathrm{Ker}(G \rightarrow G^{l})]$ in $G^{l}$ is trivial, we have an exact sequence
$$N/[N,\mathrm{Ker}(G \rightarrow G^{l})] \rightarrow G^{l} \rightarrow H^{l} \rightarrow 1,$$
and hence also the desired exact sequence
$$(N/[N,\mathrm{Ker}(G \rightarrow G^{l})])^{l} \rightarrow G^{l} \rightarrow H^{l} \rightarrow 1.$$
Thus, assertion 1 holds.
Since we have $s(\mathrm{Ker}(H \rightarrow H^{l})) \subset \mathrm{Ker}(G \rightarrow G^{l})$, assertion 2 follows from assertion 1.
\end{proof}

\begin{exam}
Let $K, O_{K}, p, K^{\rm sep}, G_{K},$ and $I_{K}$ be as in Section \ref{intro}.
Suppose that $O_{K}$ is strictly henselian and $p=0$.
Note that the Galois group $G_{K} = I_{K}$ is isomorphic to the profinite completion $\ZZ$ of $\Z$.
Let $X_{1}$ and $X_{2}$ be proper hyperbolic curves over $K$ which have good reduction.
We give an example of hyperbolic polycurve $Z$ of bad reduction over a $K$ whose pro-$l$ outer Galois representation is trivial for all but one prime $l$.
This example shows that, unlike the case of hyperbolic curves, to look at the pro-$l$ outer Galois representation for single prime number $l$ is not enough to determine whether the hyperbolic polycurve has good reduction or not.

Fix a prime number $l_{1}$.
Let $X_{1}$ and $X_{2}$ be proper hyperbolic curves over $K$ which have good reduction.
Suppose that there exist an automorphism $\iota_{2}$ of $X_{2}$ over $K$ of order $l_{1}$ and a rational point $x_{2}$ of $X_{2}$ fixed by $\iota_{2}$.
Take a geometric point $\ast_{1}$ of $X\times_{\Spec K}\Spec K^{\sep}$ and write $\Pi'_{X_{1}}$ (resp.\,$\Delta'_{X_{1}}$) for $\pi_{1}(X_{1},\ast_{1})$ (resp.\,$\Delta'_{X_{1}}$).
By Lemma \ref{lem1}, we have a canonical decomposition $\Pi'_{X_{1}}\cong \Delta'_{X_{1}}\times G_{K}$.
Take a surjective group homomorphisms $\Pi'_{X_{1}}\to \Z/l_{1}\Z$ and $G_{K} \to \Z/l_{1}\Z$.
Let $(\Pi'_{X_{1}} \cong) \Delta'_{X_{1}} \times G_{K} \to \Z/l_{1}\Z$ be the sum of these homomorphisms.
Write $X_{1}'$ for the \'etale covering space of $X_{1}$ corresponding to $\mathrm{Ker}\,(\Pi'_{X_{1}} \to \Z/l_{1}\Z)$ and $\iota_{1}$ for a generator of $\mathrm{Aut}(X_{1}'/X_{1})$.
Consider the action of $\Z/l_{1}\Z$ on $X_{2} \times_{\mathrm{Spec}\,K}X_{1}'$ induced by $(\iota_{2},\iota_{1})$, and write $Z$ for the quotient scheme of $X_{2} \times_{\mathrm{Spec}\,K}X_{1}'$ by this $\Z/l_{1}\Z$-action.
By construction, we have a Cartesian diagram
\[
\xymatrix{ 
X_{2} \times_{\mathrm{Spec}\,K}X_{1}' \ar[r] \ar[d]
& X_{1}' \ar[d]\\
Z \ar[r] 
& X_{1}.
}
\]
Let $L''$ be an algebraic closure of the function field of $X'_{1}$ and $\ast$ a geometric point of $\{x_{2}\}\times_{\Spec K} \Spec L''$.
Write $\Pi_{Z}$ (resp.\,$\Pi_{X_{1}}$; $\Delta_{X_{1}}$; $\Delta_{X_{2}}$) for the \'etale fundamental group $\pi_{1}(Z,\ast)$ (resp.\,$\pi_{1}(X_{1},\ast)$; $\pi_{1}(X_{1}\times_{\Spec K} \Spec L'',\ast)$; $\pi_{1}(X_{2}\times_{\Spec K} \Spec L'',\ast)$).
By {\cite[PROPOSITION 2.5]{Ho}}, the sequence
\begin{equation}
1 \to \Delta_{X_{2}} \to \Pi_{Z} \to \Pi_{X_{1}} \to 1
\label{newexa}
\end{equation}
is exact and $\Delta_{X_{2}}$ is isomorphic to $\pi_{1}(X_{2}\times_{\Spec K}\Spec K^{\sep},\ast)$.
Moreover, the sequence
\begin{equation}
1 \to \Delta_{X_{2}} \to \Delta_{Z} \to \Delta_{X_{1}}  \to 1.
\label{wari}
\end{equation}
is also exact.
Since the section of $X_{2} \times_{\mathrm{Spec}\,K}X_{1}' \to X_{1}'$ determined by the point $x_{2}$ is compatible with the actions of $\Z/l_{1}\Z$, we obtain a section of $Z \to X_{1}$ by taking the quotient schemes by $\Z/l_{1}\Z$.
This section defines a canonical section of $\Pi_{Z} \to \Pi_{X_{1}}$ (resp.\,$\Delta_{Z} \to \Delta_{X_{1}}$) in (\ref{newexa}) (resp.\,(\ref{wari})).
We calculate the action
\begin{equation}
\Pi_{X_{1}} \rightarrow \mathrm{Aut}(\Delta_{X_{2}})
\label{goal}
\end{equation}
induced by the section.
Write $\mathrm{Aut}_{K}(X_{2},x_{2})$ for the subgroup of the group of automorphisms of $X_{2}$ over $K$ consisting of automorphisms fixing $x_{2}$.
Write $\psi$ for the composite homomorphism of the natural surjection $\Pi_{X_{1}}\to  \Pi_{X_{1}}/\Pi_{X_{1}'}(= \langle \iota_{1} \rangle)$, two isomorphisms $\langle \iota_{1} \rangle \cong \Z/l_{1}\Z \cong \langle \iota_{2} \rangle$, the inclusion $\langle \iota_{2} \rangle \hookrightarrow \mathrm{Aut}_{K}(X_{2},x_{2})$, and the action $\mathrm{Aut}_{K}(X_{2},x_{2})\to \Delta_{X_{2}}$ defined by the fixed point $x_{2}$.
Write $\phi$ for the homomorphism $\Pi_{X_{1}} \to\mathrm{Aut}(\Delta_{X_{2}})$ defined by the splitting of the exact sequence
$$1 \to\Delta_{X_{2}} \rightarrow \pi_{1}(X_{2}\times_{\Spec K}X_{1},\ast) \to \Pi_{X_{1}} \to 1$$
(, exactness of which follows from {\cite[PROPOSITION 2.5]{Ho}},) determined by $x_{2}$.
Then $\phi$ coincides with the composite homomorphism $\Pi_{X_{1}}\to G_{K}\to \mathrm{Aut}(\Delta_{X_{2}})$, where the second homomorphism is induced by $x_{2}$.
Since $X_{2}$ has good reduction, $\phi$ is trivial.
By the construction of $Z$, the action (\ref{goal}) coincides with $\phi + \psi (= \psi)$.
We show that, for any prime number $l_{2}$, the pro-$l_{2}$ outer Galois representation $I_{K} \to \mathrm{Out}(\Delta_{Z}^{l_{2}})$ is trivial if and only if $l_{1} \neq l_{2}$.

First, we assume that $l_{1} \neq l_{2}$.
By taking pro-$l_{2}$ completion of the exact sequence (\ref{wari}), we obtain a commutative diagram with exact horizontal lines 
\[
\xymatrix{
&\Delta_{X_{2}}^{l_{2}}\ar[r] \ar[d]
&\Delta_{Z}^{l_{2}} \ar[r] \ar@{=}[d]
&\Delta_{X_{1}}^{l_{2}} \ar@{=}[d] \ar[r]
&1\\
1 \ar[r]
& \Delta \ar[r] 
& \Delta_{Z}^{l_{2}} \ar[r] 
&\Delta_{X_{1}}^{l_{2}}  \ar[r]
& 1.
}
\]
Here, we write $\Delta$ for the image of $\Delta_{X_{2}}^{l_{2}} \rightarrow  \Delta_{Z}^{l_{2}}$.
Since the composite homomorphism $\mathrm{Ker}(\Delta_{X_{1}}\rightarrow\Delta_{X_{1}}^{l_{2}}) \hookrightarrow \Delta_{X_{1}} \hookrightarrow \Pi_{X_{1}} \to \langle \iota_{2} \rangle$ is surjective by the assumption $l_{1} \neq l_{2}$, $\Delta$ is a quotient group of $(\Delta_{X_{2}})_{\langle \iota_{2} \rangle}$ by Lemma \ref{pro-l-exact}.
Write $\Pi_{Z}^{(l_{2})}$ (resp.\,$\Pi_{X_{1}}^{(l_{2})}$) for the quotient group $\Pi_{Z}/\mathrm{Ker} (\Delta_{Z} \rightarrow \Delta_{Z}^{l_{2}})$ (resp.\,$\Pi_{X_{1}}/\mathrm{Ker}(\Delta_{X_{1}} \rightarrow \Delta_{X_{1}}^{l_{2}})$).
We have a commutative diagram with exact horizontal lines
\begin{equation}
\xymatrix{
1 \ar[r]
&\Delta \ar[r] \ar@{=}[d]
&\Delta_{Z}^{l_{2}} \ar[r] \ar[d]
&\Delta_{X_{1}}^{l_{2}} \ar[d] \ar[r]
&1\\
1 \ar[r]
& \Delta \ar[r] 
& \Pi_{Z}^{(l_{2})} \ar[r] 
&\Pi_{X_{1}}^{(l_{2})}  \ar[r]
& 1.
}
\label{actiondefined}
\end{equation}
Both horizontal lines of (\ref{actiondefined}) have a section determined by the point $x_{2}$, which induces an action of $\Pi_{X_{1}}^{(l_{2})}$ on $\mathrm{Aut}(\Delta)$ compatible with (\ref{goal}).
Since $\Delta$ is a quotient group of $(\Delta_{X_{2}})_{\langle \iota_{2} \rangle}$, this action is trivial.
Hence, we have a canonical decomposition $\Pi_{Z}^{(l_{2})} \cong \Delta \times \Pi_{X_{1}}^{(l_{2})}$.
Moreover, since $X_{1}$ has good reduction, we have a canonical decomposition $\Pi_{X_{1}}^{(l_{2})} \cong \Delta_{X_{1}}^{l_{2}} \times G_{K}$ by {\cite[Proposition 1.11]{Tama1}} and the proof of Lemma \ref{lem1}.
Therefore, we can show that the outer Galois action $I_{K} \to \mathrm{Out}(\Delta_{Z}^{l_{2}})$ is trivial by using these decompositions.

Next, we assume that $l_{1} = l_{2} (= l )$ and show that the pro-$l$ outer Galois representation
$$I_{K} \rightarrow \mathrm{Out}(\Delta_{Z}^{l})$$
is nontrivial.
Since the outer action $\Delta_{X_{1}}\to \mathrm{Out}(\Delta_{X_{2}}^{l})$ factors through $\Delta_{X_{1}}\to \Delta_{X_{1}}^{l}$, we have an exact sequence $1 \to \Delta_{X_{2}}^{l} \to \Delta_{Z}^{l} \to \Delta_{X_{1}}^{l} \to 1$ by Lemma \ref{Ander}.4.
We have a commutative diagram with exact horizontal lines
\begin{equation}
\xymatrix{
1 \ar[r]
&\Delta_{X_{2}}^{l} \ar[r] \ar@{^{(}-_>}[d]
&\Pi_{Z}^{(l)} \ar[r] \ar[d]
&\Pi_{X_{1}} ^{(l)} \ar[d] \ar[r]
&1\\
1 \ar[r]
& \Delta_{Z}^{l} \ar[r] 
&\Pi_{Z}^{(l)} \ar[r] 
&G_{K} \ar[r]
& 1.
}
\label{diaglast}
\end{equation}
Since $X_{1}$ has good reduction, we have a decomposition $\Pi_{X_{1}}^{(l)} = \Delta_{X_{1}}^{l} \times Z_{\Pi_{X_{1}}^{(l)}} (\Delta_{X_{1}}^{l})$ and $Z_{\Pi_{X_{1}}^{(l)}} (\Delta_{X_{1}}^{l})$ is isomorphic to $G_{K}$ by {\cite[Proposition 1.11]{Tama1}} and the proof of Lemma \ref{lem1}.
If the outer action $I_{K} \to \mathrm{Out}\,\Delta_{Z}^{l}$ is trivial, we would have a decomposition $\Pi_{Z}^{(l)} = \Delta_{Z}^{l} \times Z_{\Pi_{Z}^{(l)}} (\Delta_{Z}^{l})$ by the proof of Lemma \ref{lem1} and the center-freeness of $\Delta_{Z}^{l}$, and the homomorphism $\Pi_{Z}^{(l)} \rightarrow \Pi_{X_{1}}^{(l)}$ would induce a surjection $Z_{\Pi_{Z}^{(l)}} (\Delta_{Z}^{l}) \rightarrow Z_{\Pi_{X_{1}}^{(l)}} (\Delta_{X_{1}}^{l})$.
Therefore, it suffices to show that the outer action $Z_{\Pi_{X_{1}}^{(l)}} (\Delta_{X_{1}}^{l})\to \mathrm{Out}\,(\Delta_{X_{2}}^{l})$ associated to the first horizontal line of the diagram (\ref{diaglast}) is nontrivial. 
We have a diagram with exact horizontal lines
\begin{equation}
\xymatrix{
1 \ar[r]
& \Delta_{X_{2}} \ar[r] \ar@{=}[d]
&\Pi_{Z} \times_{\Pi_{X_{1}}}Z_{\Pi_{X_{1}}} (\Delta_{X_{1}}) \ar[r] \ar[d]
&Z_{\Pi_{X_{1}}} (\Delta_{X_{1}}) \ar[r] \ar[d]
& 1\\
1 \ar[r]
&\Delta_{X_{2}} \ar[r]
&\Pi_{Z} \ar[r]
&\Pi_{X_{1}} \ar[r]
&1.
}
\label{diaglast2}
\end{equation}
Write $\mathfrak{X}_{1}$ for the smooth model of $X_{1}$ and $\xi_{1}$ for the generic point of the special fiber of $\mathfrak{X}_{1}$.
Let $L$ be the field of fractions of a strict henselization of the local ring of $\mathfrak{X}_{1}$ at $\xi_{1}$ in $L''$.
By Lemma \ref{lem2}, the natural homomorphism $(G_{L}:=)\mathrm{Gal}(L''/L)\to \Pi_{X_{1}}$
induces $G_{L}\to Z_{\Pi_{X_{1}}}(\Delta_{X_{1}})(\cong G_{K})$.
Since a uniformizer of $K$ is a uniformizer of $L$, $G_{L}\to Z_{\Pi_{X_{1}}}(\Delta_{X_{1}})$ is an isomorphism.
Let $L'$ be the finite extension field of $L$ defined by $\mathrm{Ker}(G_{L} \to \Pi_{X_{1}} \rightarrow \Z/l\Z)$ and $\iota$ a generator of $\mathrm{Gal}(L'/L)$.
Consider the action of $\Z/l\Z$ on $X_{2}\times_{\Spec K}\mathrm{Spec}\,L'$ induced by $(\iota_{2},\iota)$ and write $X_{2}'$ for the quotient scheme.
Write $\Pi_{X'_{2}}$ for the \'etale fundamental group $\pi_{1}(X_{2}, \ast)$.
Then we have an exact sequence
$$1 \to \Delta_{X_{2}} \to \Pi_{X'_{2}} \to G_{L}(\cong Z_{\Pi_{X_{1}}}(\Delta_{X_{1}})) \to 1.$$
Since we have a diagram
\[
\xymatrix{ 
X_{2} \times_{\mathrm{Spec}\,K}\mathrm{Spec}\,L' \ar[r] \ar[d]
& \mathrm{Spec}\,L' \ar[d]\\
X_{2} \times_{\mathrm{Spec}\,K}X_{1}' \ar[r]
& X_{1}'
}
\]
compatible with the actions of $\Z/l\Z$, we obtain a commutative diagram with exact horizontal lines
\begin{equation}
\xymatrix{
1 \ar[r]
&\Delta_{X_{2}} \ar[r] \ar@{=}[d]
&\Pi_{X_{2}'} \ar[r] \ar[d]
&G_{L} \ar[r] \ar[d]
&1\\
1 \ar[r]
& \Delta_{X_{2}} \ar[r]
&\Pi_{Z} \times_{\Pi_{X_{1}}}Z_{\Pi_{X_{1}}} (\Delta_{X_{1}}) \ar[r]
&Z_{\Pi_{X_{1}}} (\Delta_{X_{1}}) \ar[r]
& 1.
}
\label{forLout}
\end{equation}
Therefore, it suffices to show that the outer action $G_{L} \to \mathrm{Out}(\Delta_{X_{2}}^{l})$ defined by the first line of (\ref{forLout}) is nontrivial.
As the calculation of the action (\ref{goal}), we can calculate the action of $G_{L}$ on $\Delta_{X_{2}}^{l}$ determined by the rational point of $X'_{2}$ defined by $x_{2}$.
By {\cite[Theorem 1.13]{DM}}, the action of $G_{L}$ on $\Delta_{X_{2}}^{\mathrm{ab},l}$ is nontrivial, and hence the outer action $G_{L} \to \mathrm{Out}(\Delta_{X_{2}}^{l})$ is also nontrivial.
\label{ex1}
\end{exam}
\label{example}

\section{Appendix 3: Specialization isomorphisms of pro-$p'$ \'etale fundamental groups}
The contents of this section is the same as {\cite[Section 7]{Nag2}}.
For convenience of readers, the author decided to put them on this paper also.
In this section, we prove that a sort of specialization homomorphisms of pro-$p'$ \'etale fundamental groups is an isomorphism.
This fact (cf.\,Theorem \ref{isom2}) seems to be known to experts, but the author cannot find it in the literature.

Let $K$ be a discrete valuation field,  $O_{K}$ the valuation ring of $K$, $p$ ($\geq 0$) the residual characteristic of $O_{K}$, $k$ the residual field of $O_{K}$, $O_{K}^{\mathrm{h}}$ a henselization of $O_{K}$, $O_{K}^{\mathrm{sh}}$ a strict henselization of $O_{K}^{\mathrm{h}}$, $K^{\mathrm{h}}$ the field of fractions of $O_{K}^{\mathrm{h}}$, $K^{\mathrm{sh}}$ the field of fractions of $O_{K}^{\mathrm{sh}}$, and $K^{\mathrm{sep}}$ a separable closure of $K^{\mathrm{sh}}$.
For any scheme $Z$ of finite type and geometrically connected over $K$ and any geometric point $\ast$ of the scheme $Z\times_{\Spec K}\Spec K^{\sep}$, we write $\pi_{1}(Z,\ast)^{(p')}$ for the group
$$\pi_{1}(Z,\ast)/\mathrm{Ker}(\pi_{1}(Z \times_{\mathrm{Spec}\,K} \mathrm{Spec}\,K^{\mathrm{sep}}, \ast) \rightarrow \pi_{1}(Z \times_{\mathrm{Spec}\,K} \mathrm{Spec}\,K^{\mathrm{sep}}, \ast)^{p'}).$$

\begin{lem}
Let $K \subset K_{1} \subset K_{2}$ be finite extensions of fields.
Suppose that $K_{2}$ and $K_{1}$ are Galois over $K$.
Write $O_{K_{1}}$ (resp.\,$O_{K_{2}}$) for the normalization of $O_{K}$ in $K_{1}$ (resp.\,$K_{2}$).
Suppose that $O_{K_{1}}$ is a discrete valuation ring totally ramified over $O_{K}$ and that $O_{K_{2}}$ is a discrete valuation ring unramified over $O_{K_{1}}$.
Then the maximal unramified extension $K_{3}$ of $K$ in $K_{2}$ is Galois over $K$ and $K_{2} \cong K_{1}\otimes_{K} K_{3}$.
\label{DVR}
\end{lem}

\begin{proof}
Since $O_{K_{2}}$ is a discrete valuation ring, we may assume that $K$ is complete.
In this case, one can verify the assertion easily.
\end{proof}

\begin{lem}
\begin{enumerate}
\item
Let $\overline{\mathfrak{X}} \rightarrow \Spec O_{K}$ be a proper smooth morphism with geometrically connected fibers.
Let $\mathfrak{D} \subset \overline{\mathfrak{X}}$ be a normal crossing divisor of the scheme $\overline{\mathfrak{X}}$ relative to $\Spec O_{K}$.
Write $\mathfrak{X}$ (resp.\,$\overline{X}$; $X$; $\mathfrak{X}_{k}$) for the scheme $\overline{\mathfrak{X}} \setminus \mathfrak{D}$ (resp.\,$\overline{\mathfrak{X}}\times_{\Spec O_{K}}\Spec K$; $\mathfrak{X}\times_{\Spec O_{K}}\Spec K$; $\mathfrak{X}\times_{\Spec O_{K}} \Spec k$).
Since the scheme $\mathfrak{X}_{k}$ is a dense open subscheme of a connected regular scheme, $\mathfrak{X}_{k}$ is irreducible.
We write $\xi$ for the generic point of $\mathfrak{X}_{k}$.
Take a geometric point $\overline{t}$ of $X \times_{\mathrm{Spec}\,K} \mathrm{Spec}\,K^{\mathrm{sep}}$.
Consider a finite Galois \'etale covering $Y \rightarrow X$ corresponding to an open subgroup of $\pi_{1}(X,\overline{t})^{(p')}$.
Suppose that the coefficient field of $Y$ is $K$.
Then the normalization $\Spec O(\mathfrak{Y}, \xi)$ of the spectrum of the local ring $O_{\mathfrak{X},\xi}$ in $Y$ is the spectrum of a discrete valuation ring.
\item
Suppose that $O_{K}=O_{K}^{\mathrm{h}}$.
Write $\mathfrak{X}_{0}$ (resp.\,$X_{0}$) for the spectrum of the ring $O_{K}$ (resp.\,$K$).
Let
$$\mathfrak{X}_{n} \rightarrow \ldots \rightarrow \mathfrak{X}_{0}$$
be morphisms such that there exist a proper smooth morphism $\overline{\mathfrak{X}}_{i+1} \rightarrow \mathfrak{X}_{i}$ with geometrically connected fibers and a normal crossing divisor $\mathfrak{D}_{i+1} \subset \overline{\mathfrak{X}}_{i+1}$ of the scheme $\overline{\mathfrak{X}}_{i+1}$ relative to $\mathfrak{X}_{i}$ satisfying that the complement $\overline{\mathfrak{X}}_{i+1} \setminus \mathfrak{D}$ is isomorphic to $\mathfrak{X}_{i+1}$ for each $0 \leq i \leq n-1$.
Write $\overline{X}_{i}$ (resp.\,$X_{i}$) for the scheme $\overline{\mathfrak{X}}_{i}\times_{\Spec O_{K}}\Spec K$ (resp.\,$\mathfrak{X}_{i}\times_{\Spec O_{K}}\Spec K$).
Since the scheme $\mathfrak{X}_{i,k} = \mathfrak{X}_{i}\times_{\Spec O_{K}} \Spec k$ is a dense open subscheme of a connected regular scheme, $\mathfrak{X}_{k}$ is irreducible for each $0 \leq i \leq n$.
We write $\xi_{i}$ for the generic point of $\mathfrak{X}_{i,k}$ for each $0 \leq i \leq n$.
Consider a finite Galois \'etale covering $Y_{n} \rightarrow X_{n}$ corresponding to an open subgroup of $\pi_{1}(X_{n},\overline{t})^{(p')}$.
Then the normalization $\Spec O(\mathfrak{Y}_{n}, \xi_{n})$ of the spectrum of the local ring $O_{\mathfrak{X}_{n},\xi_{n}}$ in $Y_{n}$ is the spectrum of a discrete valuation ring.
\end{enumerate}
\label{irreducible}
\end{lem}

\begin{proof}
To show the assertion 1, we may assume that $O_{K}$ is strictly henselian.
Write $G$ for the automorphism group of $Y$ over $X$.
Then $Y$ is a $G$-torsor over $X$.
Since the coefficient field of $Y$ is $K$, $G$ is a finite pro-$p'$ group.
By {\cite[Expos\'e XIII, Corollaire 2.9]{SGA1}}, there exists a finite Galois \'etale morphism $\mathfrak{Z} \rightarrow \mathfrak{X}$ such that the pull-back $\mathfrak{Z} \times_{\Spec O_{K}} \Spec K^{\mathrm{sep}}$ is isomorphic to $Y\times_{\Spec K} \Spec K^{\mathrm{sep}}$ over $X\times_{\Spec K}\Spec K^{\mathrm{sep}}$.
Therefore, it suffices to show that the normalization $\Spec O(\mathfrak{Z}, \xi)$ of the spectrum of the local ring $O_{\mathfrak{X},\xi}$ in $\mathfrak{Z}$ is the spectrum of a discrete valuation ring.
Let $\overline{t'}$ be a geometric point of $\mathfrak{X}_{k}$. 
The induced homomorphism $\pi_{1}(\mathfrak{X}_{k}, \overline{t'})^{p'} \rightarrow \pi_{1}(\mathfrak{X}, \overline{t'})^{p'}$ is an isomorphism again by {\cite[Expos\'e XIII, 2.10 and Corollaire 2.9]{SGA1}}.
Therefore, the scheme $\mathfrak{Z} \times_{\Spec O_{K}} \Spec k$ is irreducible, and hence the scheme $\Spec O(\mathfrak{Z}, \xi)$ is local.

Next, we show the assertion 2.
Let $Y_{i}$ (resp.\,$\mathfrak{Y}_{i}$) be the normalization of $X_{i}$ (resp.\,$\mathfrak{X}_{i}$) in $Y_{n}$.
Since the morphism $X_{n} \rightarrow X_{i}$ is smooth, the scheme $X_{n} \times_{X_{i}}Y_{i}$ is normal.
Since the morphism $X_{n} \to X_{i}$ is flat and generically geometrically connected, the morphism $Y_{i} \rightarrow X_{i}$ is finite \'etale for each $0 \leq i \leq n$.
We will show the assertion 2 by induction on $n$.
Note that the normalization $\Spec O(\mathfrak{Y}_{0}, \xi_{0})$ of the spectrum of the local ring $O_{\mathfrak{X}_{0},\xi_{0}} = \Spec O_{K}$ in $Y_{0}$ is the spectrum of a discrete valuation ring.
Therefore, we may assume that $Y_{0} = X_{0}$.

If $n=1$, the normalization $\Spec O(\mathfrak{Y}_{1}, \xi_{1})$ of the spectrum of the local ring $O_{\mathfrak{X}_{1},\xi_{1}}$ in $Y_{1}$ is the spectrum of a discrete valuation ring by the assertion 1.
Assume that the normalization $\Spec O(\mathfrak{Y}_{n-1}, \xi_{n-1})$ of the spectrum of the local ring $O_{\mathfrak{X}_{n-1},\xi_{n-1}}$ in $Y_{n-1}$ is the spectrum of a discrete valuation ring.
Write $K(Y_{n-1})$ for the field of fractions of the scheme $Y_{n-1}$.
By applying the assertion 1 to the pair $Y_{n} \times_{Y_{n-1}} \Spec K(Y_{n-1}) \rightarrow X_{n}\times_{X_{n-1}} \Spec K(Y_{n-1}) \rightarrow \Spec K(Y_{n-1})$ and $\mathfrak{X}_{n}\times_{\mathfrak{X}_{n-1}} \Spec O(\mathfrak{Y}_{n-1}, \xi_{n-1}) \rightarrow \Spec O(\mathfrak{Y}_{n-1}, \xi_{n-1})$, we can show that the scheme $\Spec O(\mathfrak{Y}_{n}, \xi_{n})$ is local.
\end{proof}

\begin{thm}
Let $\mathfrak{X} \rightarrow \Spec O_{K}$ be a morphism satisfying the following condition:
There exists a facotrization
$$\mathfrak{X} = \mathfrak{X}_{n} \rightarrow \ldots \rightarrow \mathfrak{X}_{0} = \Spec O_{K}$$
such that there exist a proper smooth morphism $\overline{\mathfrak{X}}_{i+1} \rightarrow \mathfrak{X}_{i}$ with geometrically connected fibers and a normal crossing divisor $\mathfrak{D}_{i+1} \subset \overline{\mathfrak{X}}_{i+1}$ of the scheme $\overline{\mathfrak{X}}_{i+1}$ relative to $\mathfrak{X}_{i}$ satisfying that the complement $\overline{\mathfrak{X}}_{i+1} \setminus \mathfrak{D}$ is isomorphic to $\mathfrak{X}_{i+1}$ for each $0 \leq i \leq n-1$.
Write $X$ for the scheme $\mathfrak{X} \times_{\mathrm{Spec}\,O_{K}}\mathrm{Spec}\,K$ and take a geometric point $\overline{t}$ of $X \times_{ \mathrm{Spec}\,K} \mathrm{Spec}\,K^{\mathrm{sep}}$ over its generic point.
Then the natural homomorphism
$$\alpha: \pi_{1}(X \times_{\mathrm{Spec}\,K} \mathrm{Spec}\,K^{\mathrm{sep}}, \overline{t})^{p'} \rightarrow \pi_{1}(\mathfrak{X} \times_{\mathrm{Spec}\,O_{K}} \mathrm{Spec}\,O_{K}^{\mathrm{sh}}, \overline{t})^{p'}$$
is an isomorphism.
\label{isom2}
\end{thm}

\begin{proof}
We may and do assume $K = K^{\mathrm{sh}}$.
Since the scheme $\mathfrak{X}$ is smooth over $\mathrm{Spec}\,O_{K}$, $\mathfrak{X}$ is normal.
Since the \'etale fundamental group of the scheme $\mathrm{Spec}\,O_{K} (= \mathrm{Spec}\,O_{K}^{\mathrm{sh}})$ is trivial, for any Galois \'etale covering space $\mathfrak{Y}\to \mathfrak{X}$, the coefficient field of $\mathfrak{Y}$ is equal to $K$.
Hence, the homomorphism
$$\pi_{1}(X \times_{\mathrm{Spec}\,K} \mathrm{Spec}\,K^{\mathrm{sep}}, \overline{t}) \rightarrow \pi_{1}(\mathfrak{X} \times_{\mathrm{Spec}\,O_{K}} \mathrm{Spec}\,O_{K}^{\mathrm{sh}}, \overline{t})$$
is surjective.
Therefore, $\alpha$ is also surjective.

We prove that the homomorphism $\alpha$ is injective. 
It suffices to show that each \'etale covering space of $X \times_{\mathrm{Spec}\,K} \mathrm{Spec}\,K^{\mathrm{sep}}$ corresponding to an open subgroup of $\pi_{1}(X \times_{\mathrm{Spec}\,K} \mathrm{Spec}\,K^{\mathrm{sep}}, \overline{t})^{p'}$ is isomorphic to the pull-back of an \'etale covering space of $\mathfrak{X}$, over $X \times_{\mathrm{Spec}\,K} \mathrm{Spec}\,K^{\mathrm{sep}}$.
Since each open subgroup of $\pi_{1}(X \times_{\mathrm{Spec}\,K} \mathrm{Spec}\,K^{\mathrm{sep}}, \overline{t})^{p'}$ includes the intersection of $\pi_{1}(X \times_{\mathrm{Spec}\,K} \mathrm{Spec}\,K^{\mathrm{sep}}, \overline{t})^{p'}$ and an open normal subgroup of the group $\pi_{1}(X,\overline{t})^{(p')}$,
Theorem \ref{isom2} follows from the next lemma.
\end{proof}

\begin{lem}
Let $X, \mathfrak{X},$ and $\overline{t}$ be as in Theorem \ref{isom2}.
Suppose that $O_{K}=O_{K}^{\mathrm{sh}}$.
Let $Y$ be a Galois \'etale covering space of $X$ corresponding to an open subgroup of $\pi_{1}(X,\overline{t})^{(p')}$.
Write $K_{Y} (\subset K^{\mathrm{sep}})$ for the coefficient field of $Y$ and $e$ for the extension degree of $Y\times_{\Spec K_{Y}} \Spec K^{\mathrm{sep}} \rightarrow X\times_{\Spec K} \Spec K^{\mathrm{sep}}$,  which is prime to $p$.
Let $K'$ be the tame extension of $K_{Y}$ of degree $e$ in $K^{\mathrm{sep}}$.
Then there exists a Galois \'etale covering space $\mathfrak{X}'$ of $\mathfrak{X}$ such that the scheme $Y \times_{\Spec K_{Y}}\Spec K'$ is isomorphic to the scheme $\mathfrak{X}' \times_{\Spec O_{K}}\Spec K'$ over $X \times_{\Spec K}\Spec K'$.

\[
\xymatrix{ 
 \ar[dd] \ar@{=}[r]
\mathfrak{X}' \times_{\Spec O_{K}}\Spec K'
& Y \times_{\Spec K_{Y}}\Spec K' \ar[d] \ar[r]
& X \times_{\Spec K}\Spec K' \ar[d] \ar[r] & \Spec K' \ar[d]\\
\ar[d]  
& Y \ar[r] \ar[rd]
& X\times_{\Spec K}\Spec K_{Y} \ar[d] \ar[r] & \Spec K_{Y} \ar[d]\\
\mathfrak{X}' \times_{\Spec O_{K}}\Spec K \ar[d] \ar[rr] 
& & X \ar[d] \ar[r] & \Spec K \ar[d]\\
 \mathfrak{X}'  \ar[rr]
& 
& \mathfrak{X} \ar[r] & \Spec O_{K}.
}
\]
\label{existence}
\end{lem}

\begin{proof}
Note that the field extension $K' \supset K$ is Galois, and hence the \'etale covering $Y \times_{\Spec K_{Y}}\Spec K' \rightarrow X$ is also Galois.
By Abhyankar's lemma and Zariski-Nagata purity, the normalization $\mathfrak{Y}'$ of $\mathfrak{X} \times_{\Spec O_{K}}\Spec O_{K'}$ in the field of fractions of $Y\times_{\Spec K_{Y}}\Spec K'$ is \'etale over $\mathfrak{X} \times_{\Spec O_{K}}\Spec O_{K'}$.
Let $\xi_{X}$ (resp.\,$\xi_{X,K'}$) be the generic point of the special fiber of $\mathfrak{X}$ (resp.\,$\mathfrak{X} \times_{\Spec O_{K}}\Spec O_{K'}$).
Then the extension $O_{\mathfrak{X} \times_{\Spec O_{K}}\Spec O_{K'},\,\xi_{X,K'}} \supset O_{\mathfrak{X},\,\xi_{X}}$ is totally ramified and Galois.
Note that the normalization $\Spec O(\mathfrak{Y}',\xi_{X})$ of the scheme $\Spec O_{\mathfrak{X},\xi_{X}}$ in the field of fractions of $Y\times_{\Spec K_{Y}}\Spec K'$ is the spectrum of a discrete valuation field by Lemma \ref{irreducible}.
By Lemma \ref{DVR}, the field of fractions of the maximal unramified extension field of $O_{\mathfrak{X},\xi_{X}}$ in $O(\mathfrak{Y}',\xi_{X})$ is Galois over the function field of $\mathfrak{X}$.
We write $\mathfrak{X}'$ for the normalization of $\mathfrak{X}$ in this field.
Then the morphism $\mathfrak{X}' \rightarrow \mathfrak{X}$ is \'etale over $X$ and $\xi_{X}$.
By Zariski-Nagata purity theorem, the morphism $\mathfrak{X}' \rightarrow \mathfrak{X}$ is \'etale.
\end{proof}

\end{document}